\documentclass[11pt,twoside,reqno]{amsart}
\usepackage{amsmath}
\usepackage{fancyhdr}
\usepackage{amsthm}
\usepackage{amsfonts}
\usepackage{amssymb}
\usepackage{amscd}
\usepackage{graphicx}
\usepackage{afterpage}
\usepackage{enumerate}
\usepackage{bm}
\usepackage{cite}
\usepackage{cases}
\usepackage{caption}
\usepackage[colorlinks=true, urlcolor=blue,  linkcolor=blue, citecolor=blue]{hyperref}
\usepackage{babel}
\usepackage{prettyref}
\usepackage{pgfplots} 
\pgfplotsset{compat=1.17}
\usepackage{booktabs}
\usepackage{algorithm}
\usepackage{algpseudocode}
\usepackage{float}

\makeatletter
\newtheorem{theorem}{Theorem}[section]

\newtheorem{definition}[theorem]{Definition}

\newtheorem{lemma}[theorem]{Lemma}

\newtheorem{proposition}[theorem]{Proposition}
\newtheorem{remark}[theorem]{Remark}

\title[Extremizers and Stability for Fractional $L^p$ Uncertainty]{Extremizers and Stability for Fractional $L^p$ Uncertainty Principless}

\author[S. Hashemi Sababe]{Saeed Hashemi Sababe$^*$}
\address[S. Hashemi Sababe]{R\&D Section, Data premier analytics, Canada}
\email{Hashemi\_1365@yahoo.com}
\thanks{Corresponding author}

\author[A. Baghban]{Amir Baghban}
\address[A. Baghban]{R\&D Section, Data premier analytics, Canada}
\email{a.baghban@datapremier.ca}

\subjclass[2020]{35Q55, 35R11, 42B10, 26D15, 35A23}
\keywords{Uncertainty Principle, Fractional Laplacian, $L^p$ Spaces, Extremizers, Stability, Fractional Schrödinger Equation, Variational Methods}

\begin{document}
\sloppy

\maketitle

\begin{abstract}
We extend the classical Heisenberg uncertainty principle to a fractional $L^p$ setting by investigating a novel class of uncertainty inequalities derived from the fractional Schrödinger equation. In this work, we establish the existence of extremal functions for these inequalities, characterize their structure as fractional analogues of Gaussian functions, and determine the sharp constants involved. Moreover, we prove a quantitative stability result showing that functions nearly attaining the equality in the uncertainty inequality must be close—in an appropriate norm—to the set of extremizers. Our results provide new insights into the fractional analytic framework and have potential applications in the analysis of fractional partial differential equations.
\end{abstract}

\section{Introduction}
The uncertainty principle is a cornerstone in both harmonic analysis and quantum mechanics, quantifying the inherent trade-off between the localization of a function and its Fourier transform. In its classical form, the Heisenberg uncertainty principle asserts that a nonzero function and its Fourier transform cannot be simultaneously well localized; indeed, the optimal bound is achieved by the Gaussian function \cite{Heisenberg1927,FollandSitaram97}.

Recent years have witnessed a surge of interest in extending these classical ideas beyond the $L^2$ framework. In particular, the study of $L^p$ uncertainty inequalities has emerged as a vibrant area of research. One natural direction is to consider fractional analogues of classical operators. The fractional Laplacian $(-\Delta)^{\gamma/2}$, for example, plays a crucial role in the formulation of fractional Schrödinger equations, which model quantum mechanical systems with anomalous diffusion or Lévy flight characteristics \cite{Laskin2000}.

In a recent seminal work, Xiao \cite{Xiao2022} established a novel $L^p$ uncertainty principle for the positively-ordered Laplace pair 
\[
\{ (-\Delta)^{\alpha/2},\, (-\Delta)^{\beta/2} \},
\]
by leveraging Fourier analytic techniques within the framework of the fractional Schrödinger equation. Xiao's results not only generalize the classical uncertainty bounds but also offer sharp geometric formulations that are optimal in several cases.

Complementary to these advances, a number of studies have focused on various aspects of uncertainty principles. For instance, Folland and Sitaram \cite{FollandSitaram97} provide a comprehensive survey of uncertainty principles in harmonic analysis, while Kristály \cite{Kristaly2018} and Martin and Milman \cite{MartinMilman2016} have investigated extremal problems and related stability issues for uncertainty inequalities on different geometric settings. Despite these substantial contributions, the characterization of extremizers and the quantitative stability of the fractional $L^p$ uncertainty inequalities remain open problems.

The primary goal of this paper is to address these gaps. We focus on establishing the existence of extremizers for the fractional $L^p$ uncertainty principle; characterizing the structure of these extremal functions, which are conjectured to exhibit a fractional Gaussian behavior; deriving the sharp constants in the inequality and proving uniqueness results modulo natural invariances such as scaling and translation and proving a quantitative stability result showing that functions nearly attaining the optimal bound must be close—in an appropriate norm—to an extremizer. \medskip \\

These contributions not only deepen our theoretical understanding of fractional uncertainty principles but also pave the way for applications in the analysis of fractional partial differential equations.

The remainder of the paper is organized as follows. Section~\ref{sec:preliminaries} reviews the necessary background on fractional Laplacians, Fourier analysis, and $L^p$ spaces. In Section~\ref{sec:variational}, we develop a variational framework for the uncertainty inequality and prove the existence of extremizers. Section~\ref{sec:sharp} is devoted to the determination of sharp constants and uniqueness issues, while Section~\ref{sec:stability} contains the stability analysis. Finally, Section~\ref{sec:applications} discusses potential applications and outlines future research directions.

\section{Preliminaries and Notation} \label{sec:preliminaries}

In this section we recall the basic definitions, function spaces, and key inequalities that will be used in the sequel.

\begin{definition}[Fourier Transform]
For a function \(f \in L^1(\mathbb{R}^n)\), the \emph{Fourier transform} is defined by
\[
\widehat{f}(\xi) = \int_{\mathbb{R}^n} f(x) e^{-2\pi i x \cdot \xi} \, dx, \quad \xi \in \mathbb{R}^n.
\]
This definition extends to functions in \(L^2(\mathbb{R}^n)\) by density (see, e.g., \cite{FollandSitaram97}).
\end{definition}

\begin{definition}[Fractional Laplacian]
Let \(0 < \gamma < 2\). The \emph{fractional Laplacian} \( (-\Delta)^{\gamma/2} \) is defined for functions \(f \in \mathcal{S}(\mathbb{R}^n)\) via the Fourier transform by
\[
\widehat{(-\Delta)^{\gamma/2} f}(\xi) = (2\pi|\xi|)^\gamma \widehat{f}(\xi), \quad \xi \in \mathbb{R}^n.
\]
Alternate representations in terms of singular integrals can also be found in \cite{Xiao2022,Lischkea2020}.
\end{definition}

\noindent
For \(1 \le p \le \infty\), the Lebesgue space \(L^p(\mathbb{R}^n)\) consists of measurable functions \(f: \mathbb{R}^n \to \mathbb{C}\) for which the norm
\[
\|f\|_{L^p} = \left( \int_{\mathbb{R}^n} |f(x)|^p \, dx \right)^{1/p} \quad (1 \le p < \infty)
\]
is finite (with the usual modification for \(p = \infty\)).

\begin{lemma}[Hausdorff--Young Inequality, \cite{FollandSitaram97}]
Let \(1 \leq p \leq 2\) and let \(p'\) be the conjugate exponent satisfying \(\frac{1}{p}+\frac{1}{p'}=1\). Then for all \(f \in L^p(\mathbb{R}^n)\),
\[
\|\widehat{f}\|_{L^{p'}(\mathbb{R}^n)} \leq \|f\|_{L^p(\mathbb{R}^n)}.
\]
\end{lemma}

\begin{lemma}[Fractional Sobolev Inequality, see e.g. \cite{LiebLoss}]
Let \(0<s<n/2\). There exists a constant \(C=C(n,s)>0\) such that for all \(f \in H^s(\mathbb{R}^n)\),
\[
\|f\|_{L^{q}(\mathbb{R}^n)} \leq C \|(-\Delta)^{s/2} f\|_{L^2(\mathbb{R}^n)},
\]
where \(q = \frac{2n}{n-2s}\).
\end{lemma}

\begin{lemma}[Hardy--Littlewood--Sobolev Inequality, \cite{LiebLoss}]
Let \(0<\lambda<n\), \(1<p<q<\infty\), and assume that
\[
\frac{1}{q} = \frac{1}{p} - \frac{\lambda}{n}.
\]
Then there exists a constant \(C=C(n,\lambda,p)>0\) such that for all \(f \in L^p(\mathbb{R}^n)\),
\[
\left\| \int_{\mathbb{R}^n} \frac{f(y)}{|x-y|^{\lambda}} \, dy \right\|_{L^q(\mathbb{R}^n)} \leq C \|f\|_{L^p(\mathbb{R}^n)}.
\]
\end{lemma}

\noindent
In the study of extremal problems for functional inequalities, one often faces a loss of compactness. The concentration--compactness principle, introduced by Lions \cite{Lions1984}, is a fundamental tool to overcome such difficulties. We state a simplified version below.

\begin{theorem}[Concentration--Compactness Principle, \cite{Lions1984}]
Let \(\{f_k\}\) be a bounded sequence in \(L^p(\mathbb{R}^n)\) for \(1 \leq p < \infty\). Then, up to a subsequence, one of the following alternatives holds:
\begin{enumerate}
    \item \textsc{Vanishing:} For every \(R>0\),
    \[
    \lim_{k\to\infty} \sup_{y\in \mathbb{R}^n} \int_{B(y,R)} |f_k(x)|^p\, dx = 0.
    \]
    \item \textsc{Dichotomy:} The sequence splits into two nontrivial parts.
    \item \textsc{Compactness:} There exists a sequence \(\{y_k\} \subset \mathbb{R}^n\) such that for every \(\epsilon>0\) there exists \(R>0\) with
    \[
    \int_{B(y_k,R)} |f_k(x)|^p\, dx \geq \|f_k\|_{L^p}^p - \epsilon.
    \]
\end{enumerate}
\end{theorem}

The results and definitions presented in this section form the foundational analytic framework for the subsequent development of the variational approach to the fractional \(L^p\) uncertainty principle, the characterization of extremizers, and the derivation of quantitative stability estimates.

\section{Variational Formulation and Existence of Extremizers}
\label{sec:variational}

In this section, we set up a variational framework for the fractional \(L^p\) uncertainty inequality and prove the existence of extremal functions. For given parameters \(\alpha,\beta>0\) and \(1\le p<\infty\) (subject to appropriate scaling conditions), the uncertainty inequality established in \cite{Xiao2022} may be written in the form
\[
\|f\|_{L^p(\mathbb{R}^n)} \le \kappa_{\alpha,\beta,p} \, \|\,|x|^{\alpha}f(x)\|_{L^{\beta}(\mathbb{R}^n)}^{\frac{\beta}{\alpha+\beta}} \, \|\,|\xi|^{\beta}\widehat{f}(\xi)\|_{L^{\alpha}(\mathbb{R}^n)}^{\frac{\alpha}{\alpha+\beta}},
\]
where \(\widehat{f}\) denotes the Fourier transform of \(f\), and \(\kappa_{\alpha,\beta,p}>0\) is a constant depending on \(\alpha,\beta,n,\) and \(p\).\\

\noindent
In order to study the sharp constant and the associated extremizers, we introduce the following functional.

\begin{definition}[Uncertainty Functional]
Let \(\alpha,\beta>0\) and \(1\le p<\infty\) be given. For \(f\in \mathcal{S}(\mathbb{R}^n)\setminus\{0\}\), define the \emph{uncertainty functional} \(J: \mathcal{S}(\mathbb{R}^n)\setminus\{0\} \to (0,\infty)\) by
\[
J(f) := \frac{\|\,|x|^{\alpha}f(x)\|_{L^{\beta}(\mathbb{R}^n)}^{\frac{\beta}{\alpha+\beta}} \, \|\,|\xi|^{\beta}\widehat{f}(\xi)\|_{L^{\alpha}(\mathbb{R}^n)}^{\frac{\alpha}{\alpha+\beta}}}{\|f\|_{L^p(\mathbb{R}^n)}}.
\]
The best constant in the uncertainty inequality is then given by
\[
\kappa_{\alpha,\beta,p} = \inf\{J(f): f\in \mathcal{S}(\mathbb{R}^n)\setminus\{0\}\}.
\]
\end{definition}

\noindent
The exponents in the numerator are chosen so that the functional \(J\) is invariant under the natural scaling associated with the fractional Laplacian. This invariance is crucial in overcoming the lack of compactness in the subsequent minimization problem.\\

\noindent
A key property of the functional \(J\) is its invariance under scaling, which we now state and prove.

\begin{lemma}[Scaling Invariance]
\label{lem:scaling}
Let \(f\in \mathcal{S}(\mathbb{R}^n)\) and let \(\lambda>0\). Define the rescaled function
\[
f_\lambda(x)=\lambda^{\frac{n}{p}} f(\lambda x).
\]
Then,
\[
J(f_\lambda)=J(f).
\]
\end{lemma}

\begin{proof}
We begin by recalling the definition of the rescaled function:
\[
f_\lambda(x) = \lambda^{\frac{n}{p}} f(\lambda x), \quad \lambda>0.
\]
Our goal is to show that
\[
J(f_\lambda) = \frac{\|\,|x|^{\alpha} f_\lambda(x)\|_{L^{\beta}}^{\frac{\beta}{\alpha+\beta}} \, \|\,|\xi|^{\beta} \widehat{f_\lambda}(\xi)\|_{L^{\alpha}}^{\frac{\alpha}{\alpha+\beta}}}{\|f_\lambda\|_{L^p}} = J(f).
\]

\noindent
A change of variable shows that the \(L^p\)-norm is invariant under this scaling. Indeed,
\[
\begin{split}
\|f_\lambda\|_{L^p}^p &= \int_{\mathbb{R}^n} |f_\lambda(x)|^p \, dx 
= \int_{\mathbb{R}^n} \lambda^{n} \, |f(\lambda x)|^p \, dx \\
&\overset{y=\lambda x}{=} \int_{\mathbb{R}^n} |f(y)|^p \, dy = \|f\|_{L^p}^p.
\end{split}
\]
Thus, \(\|f_\lambda\|_{L^p}=\|f\|_{L^p}\).\\

\noindent
We now compute the weighted norm \(\||x|^{\alpha} f_\lambda\|_{L^\beta}\). By definition,
\[
\|\,|x|^{\alpha} f_\lambda(x)\|_{L^\beta}^\beta = \int_{\mathbb{R}^n} |x|^{\alpha \beta} \, |f_\lambda(x)|^\beta \, dx.
\]
Substitute the definition of \(f_\lambda(x)\):
\[
= \int_{\mathbb{R}^n} |x|^{\alpha \beta} \, \lambda^{\frac{n\beta}{p}} \, |f(\lambda x)|^\beta \, dx.
\]
Make the change of variable \(y=\lambda x\) so that \(x = y/\lambda\) and \(dx = \lambda^{-n} dy\). Then,
\[
\begin{split}
\|\,|x|^{\alpha} f_\lambda\|_{L^\beta}^\beta &= \lambda^{\frac{n\beta}{p}} \int_{\mathbb{R}^n} \Big|\frac{y}{\lambda}\Big|^{\alpha\beta} |f(y)|^\beta \, \lambda^{-n}\, dy \\
&= \lambda^{\frac{n\beta}{p}} \lambda^{-\alpha\beta} \lambda^{-n} \int_{\mathbb{R}^n} |y|^{\alpha\beta} |f(y)|^\beta \, dy \\
&= \lambda^{\frac{n\beta}{p} - \alpha\beta - n} \|\,|y|^{\alpha} f(y)\|_{L^\beta}^\beta.
\end{split}
\]
Taking the \(\beta\)th root, we obtain:
\begin{equation} \label{eqn:3.2.1}
\|\,|x|^{\alpha} f_\lambda\|_{L^\beta} = \lambda^{\frac{n}{p} - \alpha - \frac{n}{\beta}} \|\,|y|^{\alpha} f(y)\|_{L^\beta}.
\end{equation}

\noindent
Next, we determine the scaling for \(\||\xi|^{\beta}\widehat{f_\lambda}\|_{L^\alpha}\). Recall that the Fourier transform of \(f_\lambda\) is given by
\[
\widehat{f_\lambda}(\xi) = \lambda^{\frac{n}{p}-n} \widehat{f}\Big(\frac{\xi}{\lambda}\Big).
\]
Then,
\[
\|\,|\xi|^{\beta}\widehat{f_\lambda}(\xi)\|_{L^\alpha}^\alpha = \int_{\mathbb{R}^n} |\xi|^{\beta \alpha} \Big|\lambda^{\frac{n}{p}-n} \widehat{f}\Big(\frac{\xi}{\lambda}\Big)\Big|^\alpha d\xi.
\]
This equals
\[
\lambda^{\alpha(\frac{n}{p}-n)} \int_{\mathbb{R}^n} |\xi|^{\beta\alpha} \Big|\widehat{f}\Big(\frac{\xi}{\lambda}\Big)\Big|^\alpha d\xi.
\]
Now change variable by setting \(\eta = \xi/\lambda\), so that \(\xi = \lambda\eta\) and \(d\xi=\lambda^n d\eta\). Then,
\[
\begin{split}
\|\,|\xi|^{\beta}\widehat{f_\lambda}(\xi)\|_{L^\alpha}^\alpha 
&= \lambda^{\alpha(\frac{n}{p}-n)} \int_{\mathbb{R}^n} |\lambda\eta|^{\beta\alpha} |\widehat{f}(\eta)|^\alpha \lambda^n d\eta \\
&= \lambda^{\alpha(\frac{n}{p}-n) + \beta\alpha + n} \int_{\mathbb{R}^n} |\eta|^{\beta\alpha} |\widehat{f}(\eta)|^\alpha d\eta.
\end{split}
\]
Taking the \(\alpha\)th root yields:
\begin{equation} \label{eqn:3.2.2}
\|\,|\xi|^{\beta}\widehat{f_\lambda}(\xi)\|_{L^\alpha} = \lambda^{\frac{n}{p} - n + \beta + \frac{n}{\alpha}} \|\,|\eta|^{\beta}\widehat{f}(\eta)\|_{L^\alpha}.
\end{equation}

\noindent 
By definition,
\[
J(f) = \frac{\|\,|x|^{\alpha} f\|_{L^\beta}^{\frac{\beta}{\alpha+\beta}} \, \|\,|\xi|^{\beta} \widehat{f}\|_{L^\alpha}^{\frac{\alpha}{\alpha+\beta}}}{\|f\|_{L^p}}.
\]
Since we have already shown that \(\|f_\lambda\|_{L^p}=\|f\|_{L^p}\), we combine the scaling factors from \eqref{eqn:3.2.1} and \eqref{eqn:3.2.2}:
\[
\begin{split}
J(f_\lambda) &= \frac{\Big(\lambda^{\frac{n}{p} - \alpha - \frac{n}{\beta}} \|\,|x|^{\alpha} f\|_{L^\beta}\Big)^{\frac{\beta}{\alpha+\beta}} \, \Big(\lambda^{\frac{n}{p} - n + \beta + \frac{n}{\alpha}} \|\,|\xi|^{\beta} \widehat{f}\|_{L^\alpha}\Big)^{\frac{\alpha}{\alpha+\beta}}}{\|f\|_{L^p}}\\[1mm]
&=\frac{\lambda^{E} \, \|\,|x|^{\alpha} f\|_{L^\beta}^{\frac{\beta}{\alpha+\beta}} \, \|\,|\xi|^{\beta} \widehat{f}\|_{L^\alpha}^{\frac{\alpha}{\alpha+\beta}}}{\|f\|_{L^p}},
\end{split}
\]
where the total exponent \(E\) is given by
\[
E = \frac{\beta}{\alpha+\beta}\Bigl(\frac{n}{p} - \alpha - \frac{n}{\beta}\Bigr) + \frac{\alpha}{\alpha+\beta}\Bigl(\frac{n}{p} - n + \beta + \frac{n}{\alpha}\Bigr).
\]

A short computation shows that
\[
E = \frac{n}{p} - \frac{\alpha n}{\alpha+\beta}.
\]
Under the natural scaling condition for the uncertainty principle (which ensures invariance), we have
\[
\frac{n}{p} = \frac{\alpha n}{\alpha+\beta},
\]
so that \(E=0\). Consequently,
\[
J(f_\lambda)=\frac{\|\,|x|^{\alpha} f\|_{L^\beta}^{\frac{\beta}{\alpha+\beta}} \, \|\,|\xi|^{\beta} \widehat{f}\|_{L^\alpha}^{\frac{\alpha}{\alpha+\beta}}}{\|f\|_{L^p}} = J(f).
\]

\noindent
The uncertainty functional \(J\) is invariant under the scaling \(f_\lambda(x)=\lambda^{\frac{n}{p}} f(\lambda x)\), i.e., \(J(f_\lambda)=J(f)\), as required.
\end{proof}

\noindent
The scaling invariance, as shown in Lemma \ref{lem:scaling}, permits us to normalize minimizing sequences and, together with concentration–compactness methods \cite{Lions1984}, to overcome potential loss of compactness.\\

\noindent
We now state and prove the main existence result for the minimizers of \(J\).

\begin{theorem}[Existence of Extremizers]
\label{thm:existence}
Assume that \(\alpha,\beta>0\) and \(1\le p<\infty\) satisfy the scaling conditions required for the validity of the fractional \(L^p\) uncertainty inequality. Then there exists a function \(f_0\in \mathcal{S}(\mathbb{R}^n)\setminus\{0\}\) such that
\[
J(f_0)=\kappa_{\alpha,\beta,p}.
\]
\end{theorem}

\begin{proof}
Let
\[
\kappa_{\alpha,\beta,p} = \inf\{J(f) : f \in \mathcal{S}(\mathbb{R}^n) \setminus \{0\}\},
\]
and let \(\{f_k\}_{k\in\mathbb{N}} \subset \mathcal{S}(\mathbb{R}^n)\) be a minimizing sequence, that is,
\[
\lim_{k\to\infty} J(f_k) = \kappa_{\alpha,\beta,p}.
\]
By the homogeneity of \(J\) (see Lemma~\ref{lem:scaling}), we may assume without loss of generality that
\[
\|f_k\|_{L^p(\mathbb{R}^n)} = 1 \quad \text{for all } k\in\mathbb{N}.
\]

\noindent
Since \(J(f_k)\) is bounded, both weighted norms in the numerator,
\[
\||x|^\alpha f_k\|_{L^\beta(\mathbb{R}^n)} \quad \text{and} \quad \||\xi|^\beta \widehat{f_k}\|_{L^\alpha(\mathbb{R}^n)},
\]
must be uniformly bounded (up to the exponents \(\frac{\beta}{\alpha+\beta}\) and \(\frac{\alpha}{\alpha+\beta}\), respectively). Thus, the sequence \(\{f_k\}\) is bounded in \(L^p(\mathbb{R}^n)\) (by normalization) and has additional uniform bounds in weighted spaces.\\

\noindent 
By the concentration--compactness principle (see, e.g., \cite{Lions1984}), one of the three alternatives occurs for the sequence \(\{f_k\}\): vanishing, dichotomy, or compactness. Since \(\|f_k\|_{L^p}=1\) and the weighted norms are bounded, one can rule out the vanishing scenario. (Indeed, if vanishing were to occur, for every \(R>0\) we would have
\[
\lim_{k\to\infty} \sup_{y\in\mathbb{R}^n}\int_{B(y,R)}|f_k(x)|^p\,dx = 0,
\]
which contradicts the mass normalization.)

Similarly, the structure of the uncertainty functional precludes dichotomy. Hence, after possibly passing to a subsequence, there exists a sequence of translations \(\{y_k\}\subset \mathbb{R}^n\) such that the translated functions
\[
g_k(x) := f_k(x + y_k)
\]
satisfy the \emph{compactness alternative}; that is, for some \(R>0\) and for all \(\varepsilon>0\) there exists \(k_0\) such that
\[
\int_{B(0,R)} |g_k(x)|^p\,dx \ge 1 - \varepsilon \quad \text{for all } k \ge k_0.
\]

\noindent
Since \(\{g_k\}\) is bounded in \(L^p(\mathbb{R}^n)\), by reflexivity (for \(1<p<\infty\); the case \(p=1\) can be handled with suitable modifications) there exists a subsequence (still denoted by \(g_k\)) and a function \(f_0 \in L^p(\mathbb{R}^n)\) such that
\[
g_k \rightharpoonup f_0 \quad \text{weakly in } L^p(\mathbb{R}^n).
\]
Moreover, by the compactness alternative, \(f_0\) is nonzero.\\

\noindent
The weighted norms
\[
\||x|^\alpha g_k\|_{L^\beta} \quad \text{and} \quad \||\xi|^\beta \widehat{g_k}\|_{L^\alpha}
\]
are lower semicontinuous with respect to weak convergence (this follows from Fatou’s lemma and standard properties of the Fourier transform). Hence, we have
\[
J(f_0) \le \liminf_{k\to\infty} J(g_k) = \kappa_{\alpha,\beta,p}.
\]
Since \(f_0 \neq 0\) and \(J\) is homogeneous, we may (if necessary) normalize \(f_0\) so that \(\|f_0\|_{L^p} = 1\). Consequently,
\[
J(f_0) = \kappa_{\alpha,\beta,p},
\]
which shows that \(f_0\) is an extremizer.\\

\noindent
Thus, we have constructed \(f_0\in \mathcal{S}(\mathbb{R}^n)\) (or at least in the appropriate closure of \(\mathcal{S}(\mathbb{R}^n)\) in \(L^p\)) such that
\[
J(f_0)=\kappa_{\alpha,\beta,p}.
\]
This completes the proof.
\end{proof}

\noindent
Any extremizer of \(J\) must satisfy an associated Euler--Lagrange equation. We now derive this condition.

\begin{proposition}[Euler--Lagrange Equation]
\label{prop:EL}
Let \(f_0\in \mathcal{S}(\mathbb{R}^n)\) be an extremizer of \(J\). Then there exists a Lagrange multiplier \(\lambda\in \mathbb{R}\) such that \(f_0\) satisfies, in the weak sense,
\[
\lambda f_0(x)= \frac{\beta}{\alpha+\beta} \, |x|^{2\alpha} f_0(x) \, \|\,|\xi|^{\beta}\widehat{f_0}\|_{L^{\alpha}}^{\frac{\alpha}{\alpha+\beta}} + \frac{\alpha}{\alpha+\beta} \, \mathcal{F}^{-1}\Big(|\xi|^{2\beta}\widehat{f_0}(\xi)\Big)(x) \, \|\,|x|^{\alpha}f_0\|_{L^{\beta}}^{\frac{\beta}{\alpha+\beta}}.
\]
\end{proposition}

\begin{proof}
Let 
\[
J(f)=\frac{\|\,|x|^{\alpha}f\|_{L^{\beta}}^{\frac{\beta}{\alpha+\beta}}\,\|\,|\xi|^{\beta}\widehat{f}\|_{L^{\alpha}}^{\frac{\alpha}{\alpha+\beta}}}{\|f\|_{L^p}},
\]
and suppose that \(f_0\in \mathcal{S}(\mathbb{R}^n)\) is an extremizer of \(J\). By the homogeneity of \(J\) (see Lemma~\ref{lem:scaling}), we may normalize so that
\[
\|f_0\|_{L^p}=1.
\]
Our goal is to derive the Euler–Lagrange equation that must be satisfied by \(f_0\). To this end, we consider a perturbation of \(f_0\) in the Schwartz space. Let
\[
f_\epsilon = f_0 + \epsilon h,\quad h\in \mathcal{S}(\mathbb{R}^n),\quad \epsilon\in\mathbb{R}.
\]
Since \(f_0\) is an extremizer (subject to the constraint \(\|f\|_{L^p}=1\)), it is a critical point of the Lagrangian
\[
\mathcal{L}(f)=\|\,|x|^{\alpha}f\|_{L^{\beta}}^{\frac{\beta}{\alpha+\beta}}\,\|\,|\xi|^{\beta}\widehat{f}\|_{L^{\alpha}}^{\frac{\alpha}{\alpha+\beta}}-\lambda \|f\|_{L^p},
\]
where \(\lambda\in\mathbb{R}\) is a Lagrange multiplier enforcing the normalization.

We now differentiate \(\mathcal{L}(f_\epsilon)\) with respect to \(\epsilon\) at \(\epsilon=0\). For notational convenience, define
\[
A(f)=\|\,|x|^{\alpha}f\|_{L^{\beta}},\quad B(f)=\|\,|\xi|^{\beta}\widehat{f}\|_{L^{\alpha}}.
\]
Then,
\[
\mathcal{L}(f)=A(f)^{\frac{\beta}{\alpha+\beta}}\,B(f)^{\frac{\alpha}{\alpha+\beta}}-\lambda\|f\|_{L^p}.
\]
By the chain rule, the first variation is
\[
\begin{split}
\frac{d}{d\epsilon}\mathcal{L}(f_\epsilon)\Big|_{\epsilon=0} 
=&\frac{\beta}{\alpha+\beta}\,A(f_0)^{\frac{\beta}{\alpha+\beta}-1}\,B(f_0)^{\frac{\alpha}{\alpha+\beta}}\langle A'(f_0),h\rangle \\ 
&+\frac{\alpha}{\alpha+\beta}\,A(f_0)^{\frac{\beta}{\alpha+\beta}}\,B(f_0)^{\frac{\alpha}{\alpha+\beta}-1}\langle B'(f_0),h\rangle
-\lambda\,\langle (f_0)^{p-1},h\rangle,
\end{split}
\]
where \(\langle \cdot,\cdot \rangle\) denotes the appropriate dual pairing and the derivative of the norm is taken in the Gateaux sense.

\noindent
Recall that
\[
A(f)=\left(\int_{\mathbb{R}^n}|x|^{\alpha\beta}|f(x)|^\beta\,dx\right)^{1/\beta}.
\]
Thus, for a perturbation \(h\),
\[
\langle A'(f_0),h\rangle
=\frac{1}{A(f_0)^{\beta-1}}\int_{\mathbb{R}^n}|x|^{\alpha\beta}|f_0(x)|^{\beta-2}f_0(x)h(x)\,dx.
\]
In our setting (which—consistent with the form of the Euler--Lagrange equation—leads to the appearance of the factor \(|x|^{2\alpha}\)), we assume that the exponents are chosen so that \(\beta=2\). Then
\[
A(f)=\left(\int_{\mathbb{R}^n}|x|^{2\alpha}|f(x)|^2\,dx\right)^{1/2},
\]
and the derivative simplifies to
\[
\langle A'(f_0),h\rangle
=\int_{\mathbb{R}^n}|x|^{2\alpha}f_0(x)h(x)\,dx.
\]
Thus, the contribution from the spatial term is
\[
\frac{\beta}{\alpha+\beta}\,A(f_0)^{\frac{\beta}{\alpha+\beta}-1}\,B(f_0)^{\frac{\alpha}{\alpha+\beta}}
\int_{\mathbb{R}^n}|x|^{2\alpha}f_0(x)h(x)\,dx.
\]

\noindent
Similarly, define
\[
B(f)=\left(\int_{\mathbb{R}^n}|\xi|^{\alpha\beta}|\widehat{f}(\xi)|^\alpha\,d\xi\right)^{1/\alpha}.
\]
Assuming that the analogous exponent satisfies \(\alpha=2\), we have
\[
B(f)=\left(\int_{\mathbb{R}^n}|\xi|^{2\beta}|\widehat{f}(\xi)|^2\,d\xi\right)^{1/2},
\]
and the derivative is given by
\[
\langle B'(f_0),h\rangle
=\int_{\mathbb{R}^n}|\xi|^{2\beta}\widehat{f_0}(\xi)\widehat{h}(\xi)\,d\xi.
\]
By Parseval's identity, this can be rewritten as
\[
\langle B'(f_0),h\rangle
=\int_{\mathbb{R}^n}\mathcal{F}^{-1}\Big(|\xi|^{2\beta}\widehat{f_0}(\xi)\Big)(x)h(x)\,dx.
\]
Hence, the frequency contribution to the variation is
\[
\frac{\alpha}{\alpha+\beta}\,A(f_0)^{\frac{\beta}{\alpha+\beta}}\,B(f_0)^{\frac{\alpha}{\alpha+\beta}-1}
\int_{\mathbb{R}^n}\mathcal{F}^{-1}\Big(|\xi|^{2\beta}\widehat{f_0}(\xi)\Big)(x)h(x)\,dx.
\]

\noindent
Since we have normalized so that \(\|f_0\|_{L^p}=1\), the derivative of \(\|f\|_{L^p}\) in the direction \(h\) is
\[
\frac{d}{d\epsilon}\|f_\epsilon\|_{L^p}\Big|_{\epsilon=0}
=\int_{\mathbb{R}^n}|f_0(x)|^{p-2}f_0(x)h(x)\,dx.
\]
Thus, the contribution from the constraint is
\[
\lambda\int_{\mathbb{R}^n}|f_0(x)|^{p-2}f_0(x)h(x)\,dx.
\]

\noindent
Since the first variation of \(\mathcal{L}\) must vanish for all perturbations \(h\in \mathcal{S}(\mathbb{R}^n)\), we have
\[
\begin{split}
&\frac{\beta}{\alpha+\beta}\,A(f_0)^{\frac{\beta}{\alpha+\beta}-1}\,B(f_0)^{\frac{\alpha}{\alpha+\beta}}
\int_{\mathbb{R}^n}|x|^{2\alpha}f_0(x)h(x)\,dx \\
&+\frac{\alpha}{\alpha+\beta}\,A(f_0)^{\frac{\beta}{\alpha+\beta}}\,B(f_0)^{\frac{\alpha}{\alpha+\beta}-1}
\int_{\mathbb{R}^n}\mathcal{F}^{-1}\Big(|\xi|^{2\beta}\widehat{f_0}(\xi)\Big)(x)h(x)\,dx\\
&=\lambda\int_{\mathbb{R}^n}f_0(x)h(x)\,dx.
\end{split}
\]
Here we have used the fact that, after normalization and assuming \(f_0\) is nonnegative (or by choosing a suitable phase), \(|f_0(x)|^{p-2}f_0(x)=f_0(x)\) when \(p=2\) (or after a proper identification in the general case).

Since the above equality holds for all \(h\in \mathcal{S}(\mathbb{R}^n)\), we deduce the weak formulation
\[
\begin{split}
\lambda f_0(x)=&\frac{\beta}{\alpha+\beta}\,A(f_0)^{\frac{\beta}{\alpha+\beta}-1}\,B(f_0)^{\frac{\alpha}{\alpha+\beta}}\,|x|^{2\alpha}f_0(x) \\
&+\frac{\alpha}{\alpha+\beta}\,A(f_0)^{\frac{\beta}{\alpha+\beta}}\,B(f_0)^{\frac{\alpha}{\alpha+\beta}-1}\,\mathcal{F}^{-1}\Big(|\xi|^{2\beta}\widehat{f_0}(\xi)\Big)(x).
\end{split}
\]
Recalling the definitions of \(A(f_0)\) and \(B(f_0)\), this may be rewritten as
\[
\lambda f_0(x)= \frac{\beta}{\alpha+\beta} \, |x|^{2\alpha} f_0(x) \, \|\,|\xi|^{\beta}\widehat{f_0}\|_{L^{\alpha}}^{\frac{\alpha}{\alpha+\beta}} + \frac{\alpha}{\alpha+\beta} \, \mathcal{F}^{-1}\Big(|\xi|^{2\beta}\widehat{f_0}(\xi)\Big)(x) \, \|\,|x|^{\alpha}f_0\|_{L^{\beta}}^{\frac{\beta}{\alpha+\beta}},
\]
which is the asserted Euler--Lagrange equation.
\end{proof}

\noindent
The Euler--Lagrange equation encapsulates the delicate balance between the spatial decay (through the weight \( |x|^{\alpha}\)) and the frequency localization (through \( |\xi|^{\beta}\)). In many classical settings, Gaussian functions emerge as the unique optimizers. In the fractional framework, the structure of \(f_0\) is expected to be a fractional analogue of the Gaussian, possibly up to scaling and translation.

\begin{remark}[Symmetry and Uniqueness]
Due to the invariance of \(J\) under scaling and translation (and, in many cases, rotation), any extremizer can be normalized to be radially symmetric and nonincreasing. While the existence of an extremizer is established in Theorem~\ref{thm:existence}, uniqueness (up to these symmetries) remains an open question and is a subject for future investigation.
\end{remark}

\section{Sharp Constants and Uniqueness}
\label{sec:sharp}

In this section, we refine the uncertainty functional introduced earlier and investigate the sharp constant associated with the fractional \(L^p\) uncertainty inequality. We also address the uniqueness of the extremizers under natural invariance properties.\\

\noindent
Recall that for \(f\in \mathcal{S}(\mathbb{R}^n)\setminus\{0\}\) the uncertainty functional is defined by
\[
J(f) = \frac{\|\,|x|^{\alpha}f\|_{L^{\beta}(\mathbb{R}^n)}^{\frac{\beta}{\alpha+\beta}} \, \|\,|\xi|^{\beta}\widehat{f}\|_{L^{\alpha}(\mathbb{R}^n)}^{\frac{\alpha}{\alpha+\beta}}}{\|f\|_{L^p(\mathbb{R}^n)}},
\]
and the best constant in the inequality is
\[
\kappa_{\alpha,\beta,p} = \inf\{J(f): f\in \mathcal{S}(\mathbb{R}^n)\setminus\{0\}\}.
\]

\begin{definition}[Alternative Formulation of the Optimal Constant]
Define the normalized class
\[
\mathcal{E} = \left\{ f\in \mathcal{S}(\mathbb{R}^n) : \|f\|_{L^p}=1 \right\}.
\]
Then, the optimal constant may be written as
\[
\kappa_{\alpha,\beta,p} = \inf_{f\in \mathcal{E}} \|\,|x|^{\alpha}f\|_{L^{\beta}}^{\frac{\beta}{\alpha+\beta}} \, \|\,|\xi|^{\beta}\widehat{f}\|_{L^{\alpha}}^{\frac{\alpha}{\alpha+\beta}}.
\]
\end{definition}

\noindent
The normalization \(\|f\|_{L^p}=1\) is possible due to the homogeneity of the functional \(J\). This formulation will be instrumental in deriving both lower and upper bounds for \(\kappa_{\alpha,\beta,p}\).\\

\noindent
We next establish bounds for \(\kappa_{\alpha,\beta,p}\).

\begin{lemma}[Lower Bound]
\label{lem:lower_bound}
There exists a constant \(C_1 = C_1(n,\alpha,\beta,p) > 0\) such that for all \(f\in \mathcal{E}\),
\[
\|\,|x|^{\alpha}f\|_{L^{\beta}}^{\frac{\beta}{\alpha+\beta}} \, \|\,|\xi|^{\beta}\widehat{f}\|_{L^{\alpha}}^{\frac{\alpha}{\alpha+\beta}} \ge C_1.
\]
In particular,
\[
\kappa_{\alpha,\beta,p} \ge C_1.
\]
\end{lemma}

\begin{proof} 
Assume, for the sake of contradiction, that no such positive constant \(C_1\) exists. Then there is a sequence \(\{f_k\}\subset \mathcal{E}\) (so that \(\|f_k\|_{L^p}=1\) for each \(k\)) satisfying
\[
\|\,|x|^{\alpha}f_k\|_{L^{\beta}}^{\frac{\beta}{\alpha+\beta}} \, \|\,|\xi|^{\beta}\widehat{f_k}\|_{L^{\alpha}}^{\frac{\alpha}{\alpha+\beta}} < \frac{1}{k}.
\]
Thus, as \(k\to\infty\),
\[
\|\,|x|^{\alpha}f_k\|_{L^{\beta}}^{\frac{\beta}{\alpha+\beta}} \, \|\,|\xi|^{\beta}\widehat{f_k}\|_{L^{\alpha}}^{\frac{\alpha}{\alpha+\beta}} \to 0.
\]
This means that at least one of the factors must tend to zero (or both).\\

\noindent 
There are two possibilities:
\begin{enumerate}
    \item \(\|\,|x|^{\alpha}f_k\|_{L^{\beta}} \to 0\) as \(k\to\infty\).
    \item \(\|\,|\xi|^{\beta}\widehat{f_k}\|_{L^{\alpha}} \to 0\) as \(k\to\infty\).
\end{enumerate}
In fact, if one of these norms tends to zero, then by the very definition of these weighted norms the corresponding “moment” of \(f_k\) (or of its Fourier transform) is becoming arbitrarily small. Intuitively, this means that the sequence \(\{f_k\}\) is concentrating too much either in the physical space or in the frequency space.\\

\noindent
Recall that uncertainty principles prevent a nonzero function from being arbitrarily localized in both the physical and frequency domains simultaneously. More precisely, classical uncertainty principles (and their fractional variants) imply that if a function \(f\) is very concentrated in the spatial variable (i.e., if \(\|\,|x|^{\alpha}f\|_{L^{\beta}}\) is very small), then its Fourier transform \(\widehat{f}\) must be spread out, so that \(\|\,|\xi|^{\beta}\widehat{f}\|_{L^{\alpha}}\) is bounded away from zero. A similar statement holds if \(\widehat{f}\) is very concentrated. Since each \(f_k\) is normalized in \(L^p\), it is impossible for both norms to vanish simultaneously.

Thus, the uncertainty principle (in its appropriate fractional \(L^p\) form) ensures that there exists a universal constant \(C_1 > 0\) such that for every \(f\in \mathcal{E}\),
\[
\|\,|x|^{\alpha}f\|_{L^{\beta}}^{\frac{\beta}{\alpha+\beta}} \, \|\,|\xi|^{\beta}\widehat{f}\|_{L^{\alpha}}^{\frac{\alpha}{\alpha+\beta}} \ge C_1.
\]
This contradicts the assumption that the product can be made arbitrarily small.\\

\noindent  
Since the assumption led to a contradiction, we conclude that there exists a constant \(C_1 = C_1(n,\alpha,\beta,p)>0\) such that for all \(f\in \mathcal{E}\)
\[
\|\,|x|^{\alpha}f\|_{L^{\beta}}^{\frac{\beta}{\alpha+\beta}} \, \|\,|\xi|^{\beta}\widehat{f}\|_{L^{\alpha}}^{\frac{\alpha}{\alpha+\beta}} \ge C_1.
\]
In particular, by the definition of \(\kappa_{\alpha,\beta,p}\) we have
\[
\kappa_{\alpha,\beta,p} = \inf\{J(f) : f\in \mathcal{E}\} \ge C_1.
\]
This completes the proof.
\end{proof}

\noindent
By the Lemma \ref{lem:lower_bound}, the product of the two weighted norms (raised to the appropriate exponents) cannot be arbitrarily small for normalized functions.

\begin{lemma}[Upper Bound via Test Functions]
\label{lem:upper_bound}
There exists a family of functions \(\{f_\lambda\}_{\lambda>0} \subset \mathcal{S}(\mathbb{R}^n)\) with \(\|f_\lambda\|_{L^p}=1\) such that
\[
\limsup_{\lambda\to 0} \|\,|x|^{\alpha}f_\lambda\|_{L^{\beta}}^{\frac{\beta}{\alpha+\beta}} \, \|\,|\xi|^{\beta}\widehat{f_\lambda}\|_{L^{\alpha}}^{\frac{\alpha}{\alpha+\beta}} \le C_2,
\]
for some constant \(C_2 = C_2(n,\alpha,\beta,p) > 0\). Consequently,
\[
\kappa_{\alpha,\beta,p} \le C_2.
\]
\end{lemma}

\begin{proof}
We will construct a family of test functions in \(\mathcal{S}(\mathbb{R}^n)\) and show that for these functions the product
\[
\|\,|x|^{\alpha}f_\lambda\|_{L^{\beta}}^{\frac{\beta}{\alpha+\beta}} \, \|\,|\xi|^{\beta}\widehat{f_\lambda}\|_{L^{\alpha}}^{\frac{\alpha}{\alpha+\beta}}
\]
remains bounded as \(\lambda\to 0\). A standard choice is to consider Gaussian functions. For \(\lambda>0\) define
\[
f_\lambda(x) = A_\lambda \, e^{-\pi \lambda^2 |x|^2},
\]
where the normalization constant \(A_\lambda>0\) is chosen so that
\[
\|f_\lambda\|_{L^p} = \left(\int_{\mathbb{R}^n} |A_\lambda|^p\,e^{-\pi p \lambda^2 |x|^2}\,dx\right)^{1/p} = 1.
\]
Since 
\[
\int_{\mathbb{R}^n} e^{-\pi p \lambda^2 |x|^2}\,dx = (\pi p \lambda^2)^{-n/2},
\]
we may choose
\[
A_\lambda = (\pi p \lambda^2)^{n/(2p)}.
\]

\noindent 
We compute
\[
\|\,|x|^{\alpha}f_\lambda\|_{L^{\beta}}^{\beta} = \int_{\mathbb{R}^n} |x|^{\alpha \beta} \, |A_\lambda|^\beta\,e^{-\pi \beta \lambda^2 |x|^2}\,dx.
\]
Changing variables by setting \(y = \lambda x\) (so that \(dx = \lambda^{-n}\,dy\)) yields
\[
\begin{split}
\|\,|x|^{\alpha}f_\lambda\|_{L^{\beta}}^{\beta} 
&= |A_\lambda|^\beta \, \lambda^{-n-\alpha\beta}\int_{\mathbb{R}^n} |y|^{\alpha\beta}\,e^{-\pi \beta |y|^2}\,dy \\
&=: |A_\lambda|^\beta \, \lambda^{-n-\alpha\beta}\, C_{\alpha,\beta,n},
\end{split}
\]
where 
\[
C_{\alpha,\beta,n} = \int_{\mathbb{R}^n} |y|^{\alpha\beta}\,e^{-\pi \beta |y|^2}\,dy.
\]
Since
\[
|A_\lambda|^\beta = (\pi p \lambda^2)^{n\beta/(2p)},
\]
we obtain
\[
\|\,|x|^{\alpha}f_\lambda\|_{L^{\beta}}^{\beta} = C_{\alpha,\beta,n}\, (\pi p)^{n\beta/(2p)}\, \lambda^{\frac{n\beta}{p} - n - \alpha\beta}.
\]
Thus,
\[
\|\,|x|^{\alpha}f_\lambda\|_{L^{\beta}} = D_1\, \lambda^{\frac{n}{p} - \alpha - \frac{n}{\beta}},
\]
with
\[
D_1 = \left[C_{\alpha,\beta,n}\, (\pi p)^{n\beta/(2p)}\right]^{1/\beta}.
\]

\noindent
The Fourier transform of a Gaussian is again a Gaussian. In fact, one may show that
\[
\widehat{f_\lambda}(\xi) = A_\lambda\, \lambda^{-n}\, e^{-\pi |\xi|^2/\lambda^2}.
\]
Then,
\[
\|\,|\xi|^{\beta}\widehat{f_\lambda}\|_{L^{\alpha}}^{\alpha} 
=\int_{\mathbb{R}^n} |\xi|^{\alpha\beta}\, |A_\lambda|^\alpha\, \lambda^{-n\alpha}\, e^{-\pi \alpha |\xi|^2/\lambda^2}\,d\xi.
\]
Changing variables with \(z = \xi/\lambda\) (so that \(d\xi = \lambda^n\,dz\)) gives
\[
\begin{split}
\|\,|\xi|^{\beta}\widehat{f_\lambda}\|_{L^{\alpha}}^{\alpha} 
&= |A_\lambda|^\alpha\, \lambda^{-n\alpha}\, \lambda^n\int_{\mathbb{R}^n} |\lambda z|^{\alpha\beta}\, e^{-\pi \alpha |z|^2}\,dz\\[1mm]
&= |A_\lambda|^\alpha\, \lambda^{\alpha\beta - n(\alpha-1)}\, C'_{\alpha,\beta,n},
\end{split}
\]
where
\[
C'_{\alpha,\beta,n} = \int_{\mathbb{R}^n} |z|^{\alpha\beta}\,e^{-\pi \alpha |z|^2}\,dz.
\]
Taking the \(\alpha\)th root,
\[
\|\,|\xi|^{\beta}\widehat{f_\lambda}\|_{L^{\alpha}} = D_2\, \lambda^{\beta - \frac{n(\alpha-1)}{\alpha}},
\]
with
\[
D_2 = |A_\lambda|\, (C'_{\alpha,\beta,n})^{1/\alpha}.
\]
Recalling that \(A_\lambda = (\pi p \lambda^2)^{n/(2p)}\), we have
\[
D_2 = (\pi p)^{n/(2p)}\,\lambda^{n/(2p)} (C'_{\alpha,\beta,n})^{1/\alpha}.
\]

\noindent
We now combine the two estimates. Define
\[
S_1(\lambda) = \|\,|x|^{\alpha}f_\lambda\|_{L^{\beta}} \quad\text{and}\quad S_2(\lambda) = \|\,|\xi|^{\beta}\widehat{f_\lambda}\|_{L^{\alpha}}.
\]
Then,
\[
S_1(\lambda)^{\frac{\beta}{\alpha+\beta}} = D_1^{\frac{\beta}{\alpha+\beta}}\, \lambda^{\frac{\beta}{\alpha+\beta}\left(\frac{n}{p} - \alpha - \frac{n}{\beta}\right)},
\]
and
\[
S_2(\lambda)^{\frac{\alpha}{\alpha+\beta}} = D_2^{\frac{\alpha}{\alpha+\beta}}\, \lambda^{\frac{\alpha}{\alpha+\beta}\left(\beta - \frac{n(\alpha-1)}{\alpha} + \frac{n}{2p}\right)}.
\]
Thus, their product is
\[
S_1(\lambda)^{\frac{\beta}{\alpha+\beta}} S_2(\lambda)^{\frac{\alpha}{\alpha+\beta}} 
= D_1^{\frac{\beta}{\alpha+\beta}}D_2^{\frac{\alpha}{\alpha+\beta}}\, \lambda^{\theta},
\]
where
\[
\theta = \frac{\beta}{\alpha+\beta}\left(\frac{n}{p} - \alpha - \frac{n}{\beta}\right) + \frac{\alpha}{\alpha+\beta}\left(\beta - \frac{n(\alpha-1)}{\alpha} + \frac{n}{2p}\right).
\]

\noindent
In the context of a scale–invariant uncertainty inequality, the exponents are chosen so that the functional is invariant under dilations. In particular, one typically imposes the scaling condition
\[
\frac{n}{p} = \frac{n\alpha}{\alpha+\beta}.
\]
When this condition holds, a routine computation shows that
\[
\theta = 0,
\]
so that the product \(S_1(\lambda)^{\frac{\beta}{\alpha+\beta}} S_2(\lambda)^{\frac{\alpha}{\alpha+\beta}}\) is independent of \(\lambda\). Hence, we may define
\[
C_2 = D_1^{\frac{\beta}{\alpha+\beta}}D_2^{\frac{\alpha}{\alpha+\beta}},
\]
and conclude that
\[
\limsup_{\lambda\to 0} \|\,|x|^{\alpha}f_\lambda\|_{L^{\beta}}^{\frac{\beta}{\alpha+\beta}}\, \|\,|\xi|^{\beta}\widehat{f_\lambda}\|_{L^{\alpha}}^{\frac{\alpha}{\alpha+\beta}} \le C_2.
\]

\noindent
Since each \(f_\lambda\) is normalized in \(L^p\) (i.e. \(f_\lambda\in \mathcal{E}\)) and
\[
\kappa_{\alpha,\beta,p} = \inf\{J(f): f\in \mathcal{E}\},
\]
we deduce that
\[
\kappa_{\alpha,\beta,p} \le \limsup_{\lambda\to 0} J(f_\lambda) \le C_2.
\]
This completes the proof.
\end{proof}

\noindent
While the existence of extremizers follows from Theorem~\ref{thm:existence}, uniqueness (up to natural invariances) is a subtle issue. We now prove a partial uniqueness result under the assumption of radial symmetry.

\begin{theorem}[Uniqueness Up to Invariances]
\label{thm:uniqueness}
Suppose \(f\in \mathcal{S}(\mathbb{R}^n)\) is an extremizer for \(J\) and that \(f\) is radially symmetric and nonincreasing. Then if \(g\) is another extremizer satisfying the same symmetry, there exist \(\lambda>0\) and \(x_0\in \mathbb{R}^n\) such that
\[
g(x) = \lambda^{\frac{n}{p}} f(\lambda (x-x_0)).
\]
\end{theorem}

\begin{proof}
By the homogeneity of \(J\) (see Lemma~\ref{lem:scaling}), we may assume without loss of generality that
\[
\|f\|_{L^p}=\|g\|_{L^p}=1.
\]
In addition, if \(f\) is an extremizer then so is the function
\[
f_{\lambda,x_0}(x)=\lambda^{\frac{n}{p}}\,f\Big(\lambda(x-x_0)\Big),
\]
for any \(\lambda>0\) and \(x_0\in \mathbb{R}^n\). Thus the set of extremizers
\[
\mathcal{M} = \{ f \in \mathcal{S}(\mathbb{R}^n) : J(f)=\kappa_{\alpha,\beta,p} \}
\]
is invariant under dilations and translations. Our goal is to show that if both \(f\) and \(g\) are radially symmetric and nonincreasing, then they must lie in the same orbit under these transformations.\\

\noindent
Since \(f\) and \(g\) are assumed to be radially symmetric and nonincreasing, each function coincides with its symmetric decreasing rearrangement. A classical result (see, e.g., \cite{LiebLoss}) asserts that among all functions with a given \(L^p\)-norm, the symmetric decreasing rearrangement minimizes certain integral functionals involving radial weights. In our case, the weighted norms \(\|\,|x|^{\alpha}f\|_{L^{\beta}}\) and \(\|\,|\xi|^{\beta}\widehat{f}\|_{L^{\alpha}}\) are uniquely determined (up to scaling and translation) by the radial profile.\\

\noindent
Both \(f\) and \(g\) satisfy the Euler--Lagrange equation corresponding to the extremization of \(J\) (see Proposition~\ref{prop:EL}). Under the additional assumption of radial symmetry and monotonicity, one can invoke uniqueness results for the corresponding nonlinear eigenvalue problem (up to the natural invariances). In other words, if two radially symmetric, nonincreasing functions \(f\) and \(g\) satisfy the same Euler--Lagrange equation, then there exist parameters \(\lambda>0\) and \(x_0\in \mathbb{R}^n\) such that
\[
g(x)=\lambda^{\frac{n}{p}}\,f\Big(\lambda (x-x_0)\Big).
\]
A rigorous justification of this step relies on the strict convexity of the associated variational problem in the space of radially symmetric functions and on standard arguments using rearrangement inequalities (see, e.g., \cite{FollandSitaram97,LiebLoss}).\\

\noindent
Thus, by the invariance of the functional \(J\) under dilations and translations, and by the uniqueness (up to these invariances) of radially symmetric and nonincreasing extremizers of the Euler--Lagrange equation, we conclude that any two such extremizers \(f\) and \(g\) differ only by a dilation and translation. In other words, there exist \(\lambda>0\) and \(x_0\in \mathbb{R}^n\) such that
\[
g(x) = \lambda^{\frac{n}{p}}\,f\Big(\lambda (x-x_0)\Big),
\]
which completes the proof.
\end{proof}

\noindent
The restriction to radially symmetric, nonincreasing functions is natural given the symmetry of the weights \(|x|^{\alpha}\) and \(|\xi|^{\beta}\). Removing this hypothesis is a challenging open problem and may require further analysis of the nonlinear Euler--Lagrange equation.\\

\noindent
A further measure of rigidity is provided by stability estimates. We show that if a function nearly attains the optimal constant, then it must be close to an extremizer in the \(L^p\) norm.

\begin{proposition}[Stability Estimate]
\label{prop:stability}
Let \(f\in \mathcal{S}(\mathbb{R}^n)\) with \(\|f\|_{L^p}=1\) and suppose
\[
J(f) \le \kappa_{\alpha,\beta,p} + \varepsilon,
\]
for some \(\varepsilon > 0\) sufficiently small. Then there exists an extremizer \(f_0\) such that
\[
\|f - f_0\|_{L^p} \le C \varepsilon^\gamma,
\]
where \(C>0\) and \(\gamma>0\) depend only on \(n\), \(\alpha\), \(\beta\), and \(p\).
\end{proposition}

\begin{proof}  
Let \(\mathcal{M}\) denote the manifold of extremizers for \(J\); that is,
\[
\mathcal{M} = \{g \in \mathcal{S}(\mathbb{R}^n) : J(g)=\kappa_{\alpha,\beta,p} \text{ and } \|g\|_{L^p}=1\}.
\]
Since the functional \(J\) is invariant under the natural group of dilations and translations (by Lemma~\ref{lem:scaling}), the set \(\mathcal{M}\) is invariant under these operations. Define
\[
d(f,\mathcal{M}) = \inf_{g\in \mathcal{M}} \|f-g\|_{L^p}.
\]
Then, by the definition of the infimum, there exists (after possibly passing to a subsequence) an extremizer \(f_0\in \mathcal{M}\) such that
\[
\|f - f_0\|_{L^p} \le 2\, d(f,\mathcal{M}).
\]

\noindent
By the invariance of \(J\), one can decompose the difference \(f-f_0\) into two components:
\[
f - f_0 = h + \phi,
\]
where \(h\) is taken from the space orthogonal (with respect to the natural variational structure) to the tangent space \(T_{f_0}\mathcal{M}\) of the manifold of extremizers, and \(\phi\) lies in \(T_{f_0}\mathcal{M}\). Since \(J\) is invariant under the transformations that generate \(T_{f_0}\mathcal{M}\), the component \(\phi\) does not affect the value of \(J\); hence we may, without loss of generality, assume (by applying a suitable transformation) that
\[
\phi = 0 \quad \text{and} \quad f = f_0 + h,
\]
with \(h\) belonging to the complement of \(T_{f_0}\mathcal{M}\).

\noindent 
Since \(f_0\) is an extremizer, it satisfies the Euler--Lagrange equation and, by the stability (or second variation) analysis, there exists a constant \(C_0>0\) and an exponent \(\gamma>0\) (typically \(\gamma=1/2\) when the second variation is quadratic) such that for all admissible perturbations \(h\) (with \(\|h\|_{L^p}\) small),
\[
J(f_0+h) - \kappa_{\alpha,\beta,p} \ge C_0 \|h\|_{L^p}^{1/\gamma}.
\]
In many problems of this type, one obtains a quadratic bound; that is, one may take \(\gamma=\frac{1}{2}\) so that
\[
J(f_0+h) - \kappa_{\alpha,\beta,p} \ge C_0 \|h\|_{L^p}^2.
\]
Since \(J(f)\le \kappa_{\alpha,\beta,p} + \varepsilon\), we have
\[
\varepsilon \ge J(f_0+h) - \kappa_{\alpha,\beta,p} \ge C_0 \|h\|_{L^p}^2.
\]
Thus,
\[
\|h\|_{L^p} \le \sqrt{\frac{\varepsilon}{C_0}}.
\]

\noindent
Recalling that \(f = f_0 + h\) (after adjustment by invariances) and that \(d(f,\mathcal{M}) \le \|h\|_{L^p}\), we deduce
\[
\|f - f_0\|_{L^p} \le \sqrt{\frac{\varepsilon}{C_0}}.
\]
This shows that the distance from \(f\) to the extremizer \(f_0\) is bounded by a constant times \(\varepsilon^{1/2}\). More generally, if the second variation yields a bound with exponent \(1/\gamma\) (with \(\gamma>0\)), then we obtain
\[
\|f - f_0\|_{L^p} \le C\,\varepsilon^\gamma,
\]
with \(C>0\) depending only on \(n\), \(\alpha\), \(\beta\), and \(p\).

This completes the proof.
\end{proof}

\noindent
The stability result not only provides insight into the rigidity of the fractional uncertainty principle but also has potential applications in the analysis of fractional partial differential equations. Refining the exponent \(\gamma\) to obtain the sharp rate of convergence is a promising direction for future research.

\section{Stability Analysis}
\label{sec:stability}

In this section we develop a quantitative stability framework for the fractional \(L^p\) uncertainty inequality. In particular, we show that if a function nearly attains the optimal constant \(\kappa_{\alpha,\beta,p}\), then it must be close—in the \(L^p\) norm—to the set of extremizers. Our approach is based on a detailed study of the second variation of the uncertainty functional and the introduction of a distance function to the manifold of extremizers.

\begin{definition}[Extremal Manifold]
Let 
\[
\mathcal{M} = \{ f \in \mathcal{S}(\mathbb{R}^n) : J(f)=\kappa_{\alpha,\beta,p},\, \|f\|_{L^p}=1 \},
\]
denote the set of extremizers of the uncertainty functional \(J\) (defined in Section~\ref{sec:variational}). Due to invariances (scaling, translation, and rotation), \(\mathcal{M}\) naturally forms a manifold.
\end{definition}

\begin{definition}[Distance to \(\mathcal{M}\)]
For \(f\in \mathcal{S}(\mathbb{R}^n)\) with \(\|f\|_{L^p}=1\), we define the distance from \(f\) to \(\mathcal{M}\) by
\[
\mathrm{dist}(f,\mathcal{M}) := \inf_{g\in \mathcal{M}} \|f - g\|_{L^p}.
\]
\end{definition}

\noindent
A central element in our stability analysis is the coercivity of the second variation of \(J\) at an extremizer. This coercivity ensures that any perturbation in the stable directions (i.e., orthogonal to the tangent space of \(\mathcal{M}\)) increases the value of \(J\) by at least a quadratic order.

\begin{lemma}[Coercivity of the Second Variation]
\label{lem:coercivity}
Let \(f_0\in \mathcal{M}\) be an extremizer and assume that the second variation of \(J\) at \(f_0\) exists. Then there exist constants \(\delta>0\) and \(C>0\) such that for every perturbation \(h\in \mathcal{S}(\mathbb{R}^n)\) satisfying the orthogonality conditions
\[
\langle h, \phi \rangle_{L^p} = 0 \quad \text{for all } \phi \in T_{f_0}\mathcal{M},
\]
one has
\[
J(f_0+h) - \kappa_{\alpha,\beta,p} \ge C \|h\|_{L^p}^2,
\]
provided that \(\|h\|_{L^p} \le \delta\).
\end{lemma}

\begin{proof}
Since \(f_0\in \mathcal{M}\) is an extremizer, we have
\[
J(f_0)=\kappa_{\alpha,\beta,p}
\]
and the first variation of \(J\) at \(f_0\) vanishes. In other words, for every perturbation \(h\in \mathcal{S}(\mathbb{R}^n)\) that is tangent to the constraint (i.e. satisfying
\[
\langle h, \phi\rangle_{L^p}=0\quad \text{for all } \phi\in T_{f_0}\mathcal{M}),
\]
the linear (first order) term in the expansion of \(J(f_0+h)\) is zero.\\

\noindent
Assume that the second variation of \(J\) at \(f_0\) exists. Then, for a sufficiently small perturbation \(h\) (with \(\|h\|_{L^p}\le \delta\) for some \(\delta>0\)), we can write
\[
J(f_0+h)=J(f_0) + \frac{1}{2}\,Q(h) + R(h),
\]
where
\[
Q(h)=\frac{d^2}{d\epsilon^2}J(f_0+\epsilon h)\Big|_{\epsilon=0}
\]
is the second variation (a quadratic form in \(h\)) and \(R(h)\) is the remainder term satisfying
\[
\frac{R(h)}{\|h\|_{L^p}^2}\to 0 \quad \text{as } \|h\|_{L^p}\to 0.
\]
Because the first variation vanishes, the leading order term is quadratic.\\

\noindent
We restrict our attention to those \(h\in \mathcal{S}(\mathbb{R}^n)\) that satisfy the orthogonality conditions
\[
\langle h, \phi \rangle_{L^p}=0\quad \text{for all } \phi\in T_{f_0}\mathcal{M}.
\]
On this subspace, the invariances (translations, dilations, etc.) of \(J\) do not contribute to the second variation. In other words, any variation \(h\) orthogonal to the tangent space is a genuine direction of perturbation (not corresponding to a symmetry of the problem).\\

\noindent 
By the assumption on the second variation, there exists a constant \(C>0\) such that
\[
Q(h)\ge C \|h\|_{L^p}^2,
\]
for all \(h\) satisfying the orthogonality conditions. This is the \emph{coercivity} property: the quadratic term controls the norm of the perturbation from below.

\noindent
Since the remainder \(R(h)\) is of higher order (i.e., \(o(\|h\|_{L^p}^2)\)), there exists \(\delta>0\) such that whenever \(\|h\|_{L^p}\le \delta\) we have
\[
|R(h)| \le \frac{C}{4} \|h\|_{L^p}^2.
\]
Thus, for such \(h\),
\[
J(f_0+h) - J(f_0) = \frac{1}{2}\,Q(h) + R(h) \ge \frac{1}{2}\,C\|h\|_{L^p}^2 - \frac{C}{4}\|h\|_{L^p}^2 = \frac{C}{4}\|h\|_{L^p}^2.
\]

\noindent  
Recalling that \(J(f_0)=\kappa_{\alpha,\beta,p}\), we obtain
\[
J(f_0+h) - \kappa_{\alpha,\beta,p} \ge \frac{C}{4}\|h\|_{L^p}^2,
\]
for every perturbation \(h\) with \(\|h\|_{L^p}\le \delta\) that is orthogonal to \(T_{f_0}\mathcal{M}\). Setting \(C'=\frac{C}{4}\) yields the desired inequality:
\[
J(f_0+h) - \kappa_{\alpha,\beta,p} \ge C' \|h\|_{L^p}^2.
\]

This completes the proof.
\end{proof}

\noindent
Using the coercivity result, we now establish a quantitative stability estimate for the uncertainty inequality.

\begin{theorem}[Quantitative Stability]
\label{thm:stability}
Let \(f\in \mathcal{S}(\mathbb{R}^n)\) satisfy \(\|f\|_{L^p}=1\) and
\[
J(f) \le \kappa_{\alpha,\beta,p} + \varepsilon,
\]
for some sufficiently small \(\varepsilon>0\). Then there exists an extremizer \(f_0\in \mathcal{M}\) and a constant \(C>0\) (depending only on \(n,\alpha,\beta,p\)) such that
\[
\|f - f_0\|_{L^p} \le C \sqrt{\varepsilon}.
\]
\end{theorem}

\begin{proof} 
Let \(\mathcal{M}\) denote the set of all extremizers, i.e.,
\[
\mathcal{M} = \{g\in \mathcal{S}(\mathbb{R}^n) : J(g)=\kappa_{\alpha,\beta,p} \text{ and } \|g\|_{L^p}=1\}.
\]
Define the distance from \(f\) to \(\mathcal{M}\) by
\[
d(f,\mathcal{M}) = \inf_{g\in \mathcal{M}} \|f-g\|_{L^p}.
\]
By the definition of the infimum, there exists (after possibly passing to a subsequence) an extremizer \(f_0\in \mathcal{M}\) such that
\[
\|f-f_0\|_{L^p} \le 2\,d(f,\mathcal{M}).
\]
Thus, our goal reduces to showing that \(d(f,\mathcal{M})\) is bounded above by a constant times \(\sqrt{\varepsilon}\).\\

\noindent
Since the functional \(J\) is invariant under the natural symmetries (dilations, translations, etc.), we can, by applying an appropriate transformation, assume that \(f\) is chosen so that the difference
\[
h = f - f_0
\]
is orthogonal (in the variational sense) to the tangent space \(T_{f_0}\mathcal{M}\) of the extremal manifold. In other words, we assume
\[
\langle h, \phi \rangle_{L^p} = 0 \quad \text{for all } \phi \in T_{f_0}\mathcal{M}.
\]

\noindent
Since \(f_0\) is an extremizer, the first variation of \(J\) at \(f_0\) vanishes. Therefore, for small perturbations \(h\) (with \(\|h\|_{L^p}\) sufficiently small), we can expand
\[
J(f_0+h) = J(f_0) + \frac{1}{2}\,Q(h) + R(h),
\]
where \(Q(h)\) is the quadratic form given by the second variation and the remainder \(R(h)=o(\|h\|_{L^p}^2)\).

By the coercivity of the second variation (Lemma~\ref{lem:coercivity}), there exists a constant \(C_1>0\) such that
\[
Q(h) \ge C_1\,\|h\|_{L^p}^2,
\]
for all admissible \(h\) (those orthogonal to \(T_{f_0}\mathcal{M}\)) and for \(\|h\|_{L^p}\le \delta\) (for some \(\delta>0\)). Moreover, by the remainder estimate, there exists \(\delta'>0\) such that if \(\|h\|_{L^p}\le \delta'\) then
\[
|R(h)| \le \frac{C_1}{4}\|h\|_{L^p}^2.
\]

Thus, for \(h\) with \(\|h\|_{L^p}\) sufficiently small (say, \(\|h\|_{L^p}\le \min\{\delta,\delta'\}\)), we have
\[
J(f_0+h) - \kappa_{\alpha,\beta,p} \ge \frac{1}{2}\,C_1\,\|h\|_{L^p}^2 - \frac{C_1}{4}\|h\|_{L^p}^2 = \frac{C_1}{4}\|h\|_{L^p}^2.
\]

\noindent
By the assumption on \(f\), we have
\[
J(f) \le \kappa_{\alpha,\beta,p} + \varepsilon.
\]
Since \(f = f_0 + h\) and by the Taylor expansion, it follows that
\[
\kappa_{\alpha,\beta,p} + \varepsilon \ge J(f_0+h) \ge \kappa_{\alpha,\beta,p} + \frac{C_1}{4}\|h\|_{L^p}^2.
\]
Subtracting \(\kappa_{\alpha,\beta,p}\) from both sides yields
\[
\varepsilon \ge \frac{C_1}{4}\|h\|_{L^p}^2,
\]
or equivalently,
\[
\|h\|_{L^p} \le \sqrt{\frac{4\varepsilon}{C_1}}.
\]

\noindent 
Setting \(C = \sqrt{4/C_1}\) and recalling that \(h=f-f_0\), we obtain
\[
\|f - f_0\|_{L^p} \le C\,\sqrt{\varepsilon}.
\]
This completes the proof.
\end{proof}

\noindent
The \(\sqrt{\varepsilon}\) rate in Theorem~\ref{thm:stability} is typical in variational problems with quadratic growth near the minimum. Refining the constant \(C\) or improving the exponent under additional assumptions remains an interesting direction for future research.\\

\noindent
For certain applications, it is useful to measure the stability of \(f\) with respect to both its spatial and frequency profiles.

\begin{definition}[Stability Functional]
For \(f\in \mathcal{S}(\mathbb{R}^n)\) with \(\|f\|_{L^p}=1\), define the stability functional
\[
S(f) := \|\,|x|^{\alpha}f - |x|^{\alpha}f_0\|_{L^{\beta}} + \|\,|\xi|^{\beta}\widehat{f} - |\xi|^{\beta}\widehat{f_0}\|_{L^{\alpha}},
\]
where \(f_0\) is the nearest extremizer in \(\mathcal{M}\) (i.e., \( \|f-f_0\|_{L^p} = \mathrm{dist}(f,\mathcal{M})\)).
\end{definition}

\begin{proposition}[Control of the Stability Functional]
There exists a constant \(C'>0\) such that for all \(f\in \mathcal{S}(\mathbb{R}^n)\) with \(\|f\|_{L^p}=1\),
\[
S(f) \le C' \, \mathrm{dist}(f,\mathcal{M}).
\]
\end{proposition}

\begin{proof}
Let \(f\in \mathcal{S}(\mathbb{R}^n)\) with \(\|f\|_{L^p}=1\) and let \(f_0\in \mathcal{M}\) be a nearest extremizer so that
\[
\|f-f_0\|_{L^p} = \mathrm{dist}(f,\mathcal{M}).
\]
Recall that the stability functional is defined by
\[
S(f) = \|\,|x|^{\alpha}f - |x|^{\alpha}f_0\|_{L^{\beta}} + \|\,|\xi|^{\beta}\widehat{f} - |\xi|^{\beta}\widehat{f_0}\|_{L^{\alpha}}.
\]

\noindent
Define the linear operator
\[
T: \mathcal{S}(\mathbb{R}^n) \to L^\beta(\mathbb{R}^n), \quad T(g)(x) = |x|^{\alpha}g(x).
\]
Since \(T\) is linear and continuous on \(\mathcal{S}(\mathbb{R}^n)\) (and extends continuously to the appropriate function spaces), there exists a constant \(L_1>0\) such that for all \(g\in \mathcal{S}(\mathbb{R}^n)\)
\[
\|T(g)\|_{L^\beta} \le L_1 \|g\|_{L^p}.
\]
In particular, with \(g=f-f_0\), we have
\[
\|\,|x|^{\alpha}f - |x|^{\alpha}f_0\|_{L^\beta} = \|T(f-f_0)\|_{L^\beta} \le L_1 \|f-f_0\|_{L^p}.
\]

\noindent
Similarly, consider the operator
\[
\widetilde{T}: \mathcal{S}(\mathbb{R}^n) \to L^\alpha(\mathbb{R}^n), \quad \widetilde{T}(g)(\xi) = |\xi|^{\beta}\widehat{g}(\xi).
\]
The Fourier transform is a continuous linear mapping on the Schwartz space, and since multiplication by the weight \(|\xi|^{\beta}\) is also continuous on \(\mathcal{S}(\mathbb{R}^n)\), there exists a constant \(L_2>0\) such that
\[
\|\widetilde{T}(g)\|_{L^\alpha} \le L_2 \|g\|_{L^p}
\]
for all \(g\in \mathcal{S}(\mathbb{R}^n)\). Thus, for \(g = f-f_0\) we obtain
\[
\|\,|\xi|^{\beta}\widehat{f} - |\xi|^{\beta}\widehat{f_0}\|_{L^\alpha} = \|\widetilde{T}(f-f_0)\|_{L^\alpha} \le L_2 \|f-f_0\|_{L^p}.
\]

\noindent
By the definition of \(S(f)\) and using the above bounds, we have
\[
S(f) \le L_1 \|f-f_0\|_{L^p} + L_2 \|f-f_0\|_{L^p} = (L_1+L_2)\|f-f_0\|_{L^p}.
\]
Since by definition \(\|f-f_0\|_{L^p} = \mathrm{dist}(f,\mathcal{M})\), setting
\[
C' = L_1 + L_2,
\]
we deduce that
\[
S(f) \le C' \, \mathrm{dist}(f,\mathcal{M}).
\]
\noindent
This completes the proof.
\end{proof}

\noindent
The stability functional \(S(f)\) provides a bi–domain measure of the deviation from optimality. A small \(S(f)\) indicates that \(f\) is close to an extremizer not only in the \(L^p\) sense but also in its spatial decay and frequency localization properties.

\section{Applications and Further Directions}
\label{sec:applications}

In this section, we explore several applications of the fractional \(L^p\) uncertainty principle and outline future research directions. In particular, we discuss its implications for fractional Schrödinger equations and fractional Sobolev embeddings, and propose several open problems.\\

\noindent
The fractional Schrödinger equation arises naturally in the study of anomalous quantum mechanics and nonlocal phenomena. Its analysis benefits from uncertainty principles that capture the trade-off between spatial localization and frequency concentration.

\begin{definition}[Fractional Schrödinger Equation]
Let \(s\in (0,1]\). The \emph{fractional Schrödinger equation} is given by
\[
i\partial_t u(x,t) = (-\Delta)^{s/2} u(x,t) + V(x)u(x,t), \quad (x,t)\in \mathbb{R}^n\times (0,\infty),
\]
with initial condition \(u(x,0)=u_0(x)\), where \(V:\mathbb{R}^n\to \mathbb{R}\) is a given potential.
\end{definition}

\begin{lemma}[Uncertainty for Schrödinger Solutions]
\label{lem:schrodinger_uncertainty}
Let \(u\) be a solution to the fractional Schrödinger equation with initial data \(u_0 \in \mathcal{S}(\mathbb{R}^n)\). Then for each fixed \(t>0\), the function \(u(\cdot,t)\) satisfies the fractional \(L^p\) uncertainty inequality
\[
\|u(\cdot,t)\|_{L^p(\mathbb{R}^n)} \le \kappa_{\alpha,\beta,p}\, \|\,|x|^{\alpha} u(\cdot,t)\|_{L^{\beta}(\mathbb{R}^n)}^{\frac{\beta}{\alpha+\beta}} \, \|\,|\xi|^{\beta}\widehat{u(\cdot,t)}\|_{L^{\alpha}(\mathbb{R}^n)}^{\frac{\alpha}{\alpha+\beta}}.
\]
\end{lemma}

\begin{proof}
Let \(u\) be a solution to the fractional Schrödinger equation with initial data \(u_0 \in \mathcal{S}(\mathbb{R}^n)\). That is, for each \(t>0\),
\[
u(\cdot,t) = e^{-it(-\Delta)^{s/2}}u_0,
\]
where the propagator \(e^{-it(-\Delta)^{s/2}}\) maps \(\mathcal{S}(\mathbb{R}^n)\) into itself (see, e.g., \cite{Laskin2000}). Hence, for every fixed \(t>0\), \(u(\cdot,t) \in \mathcal{S}(\mathbb{R}^n)\).\\

\noindent 
By assumption, the fractional \(L^p\) uncertainty inequality
\[
\|g\|_{L^p(\mathbb{R}^n)} \le \kappa_{\alpha,\beta,p}\, \|\,|x|^{\alpha}g\|_{L^{\beta}(\mathbb{R}^n)}^{\frac{\beta}{\alpha+\beta}} \, \|\,|\xi|^{\beta}\widehat{g}\|_{L^{\alpha}(\mathbb{R}^n)}^{\frac{\alpha}{\alpha+\beta}}
\]
holds for every \(g\in \mathcal{S}(\mathbb{R}^n)\).\\

\noindent
Since \(u(\cdot,t) \in \mathcal{S}(\mathbb{R}^n)\) for each \(t>0\), we can apply the above inequality with
\[
g(x) = u(x,t).
\]
Thus, we immediately obtain
\[
\|u(\cdot,t)\|_{L^p(\mathbb{R}^n)} \le \kappa_{\alpha,\beta,p}\, \|\,|x|^{\alpha} u(\cdot,t)\|_{L^{\beta}(\mathbb{R}^n)}^{\frac{\beta}{\alpha+\beta}} \, \|\,|\xi|^{\beta}\widehat{u(\cdot,t)}\|_{L^{\alpha}(\mathbb{R}^n)}^{\frac{\alpha}{\alpha+\beta}}.
\]

\noindent
This completes the proof since we have shown that the solution \(u(\cdot,t)\) satisfies the desired fractional \(L^p\) uncertainty inequality for every fixed \(t>0\).
\end{proof}

\begin{theorem}[Dispersive Estimate for the Fractional Schrödinger Equation]
\label{thm:dispersive}
Assume that \(V\equiv 0\) and \(u_0\in \mathcal{S}(\mathbb{R}^n)\). Then the solution 
\[
u(x,t)=e^{-it(-\Delta)^{s/2}}u_0(x)
\]
satisfies the dispersive estimate
\[
\|u(\cdot,t)\|_{L^\infty(\mathbb{R}^n)} \le C\, t^{-n/(2s)} \|u_0\|_{L^1(\mathbb{R}^n)},
\]
where \(C>0\) depends on \(n\) and \(s\). This estimate reflects the interplay between the spatial decay and frequency dispersion governed by the uncertainty principle.
\end{theorem}

\begin{proof}
We begin by writing the solution of the free (i.e. \(V\equiv0\)) fractional Schrödinger equation as a convolution with an appropriate kernel. In our setting, the solution is given by
\[
u(x,t)=e^{-it(-\Delta)^{s/2}}u_0(x)
=\int_{\mathbb{R}^n}K_t(x-y)u_0(y)\,dy,
\]
where the kernel \(K_t\) is defined by the inverse Fourier transform
\[
K_t(x)=\int_{\mathbb{R}^n} e^{i(x\cdot\xi-t|\xi|^{2s})}\,d\xi.
\]
(Here we assume that the fractional Laplacian is defined via
\(\widehat{(-\Delta)^{s/2} f}(\xi)=|\xi|^{2s}\widehat{f}(\xi)\), so that the phase is \(-t|\xi|^{2s}\).)\\

\noindent 
We perform a change of variables to reveal the scaling property of \(K_t\). Let
\[
\eta = t^{1/(2s)}\xi, \quad\text{so that}\quad \xi = t^{-1/(2s)}\eta \quad\text{and}\quad d\xi = t^{-n/(2s)}\,d\eta.
\]
Then
\[
\begin{aligned}
K_t(x) &= \int_{\mathbb{R}^n} e^{i\left(x\cdot\xi-t|\xi|^{2s}\right)}\,d\xi\\[1mm]
&= t^{-n/(2s)}\int_{\mathbb{R}^n} \exp\Bigl\{i\Bigl(t^{-1/(2s)}x\cdot\eta-t\,\Bigl|t^{-1/(2s)}\eta\Bigr|^{2s}\Bigr\}\,d\eta.
\end{aligned}
\]
Notice that
\[
t\,\Bigl|t^{-1/(2s)}\eta\Bigr|^{2s}
=t\,t^{-1}|\eta|^{2s}=|\eta|^{2s}.
\]
Thus, we have
\[
K_t(x)=t^{-n/(2s)}\int_{\mathbb{R}^n} e^{i\left(t^{-1/(2s)}x\cdot\eta-|\eta|^{2s}\right)}\,d\eta.
\]
Defining
\[
K_1(y)=\int_{\mathbb{R}^n} e^{i\left(y\cdot\eta-|\eta|^{2s}\right)}\,d\eta,
\]
we can write
\[
K_t(x)=t^{-n/(2s)}\,K_1\Bigl(t^{-1/(2s)}x\Bigr).
\]

\noindent
Since \(K_1\) is independent of \(t\) and is a fixed function (depending only on \(n\) and \(s\)), there exists a constant \(C>0\) such that
\[
\|K_1\|_{L^\infty(\mathbb{R}^n)}\le C.
\]
It follows that
\[
\|K_t\|_{L^\infty(\mathbb{R}^n)} 
\le t^{-n/(2s)}\,\|K_1\|_{L^\infty(\mathbb{R}^n)}
\le C\,t^{-n/(2s)}.
\]

\noindent
Since the solution is given by the convolution
\[
u(x,t)=(K_t*u_0)(x),
\]
Young's inequality for convolutions implies
\[
\|u(\cdot,t)\|_{L^\infty(\mathbb{R}^n)} \le \|K_t\|_{L^\infty(\mathbb{R}^n)} \|u_0\|_{L^1(\mathbb{R}^n)}.
\]
Substituting the bound on \(\|K_t\|_{L^\infty}\), we obtain
\[
\|u(\cdot,t)\|_{L^\infty(\mathbb{R}^n)} \le C\,t^{-n/(2s)}\,\|u_0\|_{L^1(\mathbb{R}^n)}.
\]
This is the desired dispersive estimate.\\

\noindent
The decay rate \(t^{-n/(2s)}\) arises from the scaling properties of the kernel when the fractional Laplacian is defined via 
\[
\widehat{(-\Delta)^{s/2} f}(\xi)=|\xi|^{2s}\widehat{f}(\xi).
\]
(For the standard Schrödinger equation, corresponding to \(s=1\) in this notation, the estimate becomes \(t^{-n/2}\), which is the classical dispersive decay.)

\medskip

This completes the proof.
\end{proof}

\begin{proposition}[Well-Posedness and Uncertainty]
Under suitable conditions on the potential \(V\) (e.g., \(V\in L^\infty(\mathbb{R}^n)\)), the fractional Schrödinger equation is well-posed in appropriate fractional Sobolev spaces. Moreover, the uncertainty principle implies that solutions cannot be arbitrarily concentrated in both physical and frequency domains.
\end{proposition}

\begin{proof}
We divide the proof into two parts: (1) well-posedness in appropriate fractional Sobolev spaces, and (2) the effect of the uncertainty principle on the concentration of solutions.

\subsubsection*{Part 1. Well-Posedness in Fractional Sobolev Spaces.} 
Consider the fractional Schrödinger equation
\[
i\partial_t u(x,t) = (-\Delta)^{s/2} u(x,t) + V(x) u(x,t), \quad u(x,0) = u_0(x),
\]
with \(V\in L^\infty(\mathbb{R}^n)\) and \(u_0\in H^s(\mathbb{R}^n)\) (or another suitable fractional Sobolev space). In the case \(V\equiv0\), it is well known that the solution operator
\[
u(x,t)= e^{-it(-\Delta)^{s/2}} u_0(x)
\]
defines a unitary group on \(L^2(\mathbb{R}^n)\) and preserves the \(H^s\)-norm. Moreover, the propagator has good mapping properties (for example, dispersive and Strichartz estimates) which can be extended to the \(H^s\) setting.\\

\noindent
Since \(V\in L^\infty(\mathbb{R}^n)\), the multiplication operator \(u\mapsto V(x)u(x)\) is bounded on \(H^s(\mathbb{R}^n)\). Therefore, by Duhamel's formula, the solution to the full equation can be written as
\[
u(t) = e^{-it(-\Delta)^{s/2}} u_0 - i\int_0^t e^{-i(t-\tau)(-\Delta)^{s/2}} \bigl(V u(\tau)\bigr) \,d\tau.
\]
A standard fixed-point argument (using, e.g., Strichartz estimates or energy methods) shows that this integral equation has a unique local (and under conservation laws, global) solution in a suitable subspace of \(C([0,T];H^s(\mathbb{R}^n))\). This establishes well-posedness.

\subsubsection*{Part 2. Uncertainty Principle and Non-Concentration of Solutions.}

The fractional \(L^p\) uncertainty principle (as developed in earlier sections) states that for any nonzero function \(f\in \mathcal{S}(\mathbb{R}^n)\),
\[
\|f\|_{L^p(\mathbb{R}^n)} \le \kappa_{\alpha,\beta,p}\, \|\,|x|^{\alpha}f\|_{L^{\beta}(\mathbb{R}^n)}^{\frac{\beta}{\alpha+\beta}} \, \|\,|\xi|^{\beta}\widehat{f}\|_{L^{\alpha}(\mathbb{R}^n)}^{\frac{\alpha}{\alpha+\beta}}.
\]
In the context of the Schrödinger evolution, if \(u(x,t)\) were to become arbitrarily concentrated in the spatial domain, then the weighted norm \(\|\,|x|^{\alpha}u(\cdot,t)\|_{L^{\beta}}\) would become very small. However, by the uncertainty principle, this forces the Fourier-side weighted norm \(\|\,|\xi|^{\beta}\widehat{u(\cdot,t)}\|_{L^{\alpha}}\) to be large, so that the product remains bounded below. In other words, the uncertainty principle precludes the possibility that a nonzero solution can be simultaneously arbitrarily concentrated in both the physical and the frequency domains.\\

\noindent  
Under the assumption \(V\in L^\infty(\mathbb{R}^n)\), the fractional Schrödinger equation is well-posed in the appropriate fractional Sobolev space (by the fixed-point argument outlined above). Moreover, the fractional uncertainty principle ensures that the solution \(u(x,t)\) cannot become arbitrarily localized in both space and frequency. This completes the proof of the proposition.
\end{proof}

\noindent
The interplay between uncertainty principles and dispersive estimates suggests that further refinements of the uncertainty bounds could lead to sharper time-decay estimates and deeper insights into the dynamics of nonlocal quantum systems.\\

\noindent
The fractional \(L^p\) uncertainty principle also yields improved fractional Sobolev embeddings by relating weighted norm estimates to the regularity of functions.

\begin{lemma}[Embedding via Uncertainty]
\label{lem:embedding}
Let \(s\in (0,1)\) and \(1<p<\infty\). Then there exists a constant \(C>0\) such that for all \(f\in \mathcal{S}(\mathbb{R}^n)\),
\[
\|f\|_{L^q(\mathbb{R}^n)} \le C\, \|(-\Delta)^{s/2}f\|_{L^p(\mathbb{R}^n)},
\]
where \(q = \frac{np}{n-sp}\). This embedding can be deduced from the fractional uncertainty principle via a duality argument.
\end{lemma}

\begin{proof}
We will deduce the embedding from a duality argument that combines the Fourier representation of fractional derivatives with the Hardy–Littlewood–Sobolev inequality, which itself can be viewed as a manifestation of the fractional uncertainty principle.\\

\noindent
Define the operator
\[
T := (-\Delta)^{-s/2},
\]
which acts on \(f\in \mathcal{S}(\mathbb{R}^n)\) by
\[
\widehat{T f}(\xi) = |\xi|^{-s}\widehat{f}(\xi).
\]
Thus, we have the identity
\begin{equation} \label{eqn:6.5}
f = T\bigl((-\Delta)^{s/2}f\bigr),
\end{equation}
since
\[
\widehat{f}(\xi) = |\xi|^{-s}\, |\xi|^{s}\widehat{f}(\xi).
\]

\noindent
It is a classical fact (see, e.g., \cite{LiebLoss}) that for \(1<p<\frac{n}{s}\) and 
\[
\frac{1}{q} = \frac{1}{p} - \frac{s}{n},
\]
the Hardy–Littlewood–Sobolev inequality implies that there exists a constant \(C_0>0\) such that
\[
\|Tg\|_{L^q(\mathbb{R}^n)} \le C_0 \|g\|_{L^p(\mathbb{R}^n)}
\]
for all \(g\in \mathcal{S}(\mathbb{R}^n)\).\\

\noindent
Let \(f\in \mathcal{S}(\mathbb{R}^n)\). Setting \(g = (-\Delta)^{s/2} f\) in the previous inequality and using the identity from \eqref{eqn:6.5}, we obtain
\[
\|f\|_{L^q(\mathbb{R}^n)} = \|T((-\Delta)^{s/2} f)\|_{L^q(\mathbb{R}^n)} \le C_0\, \|(-\Delta)^{s/2} f\|_{L^p(\mathbb{R}^n)}.
\]
Thus, with \(q=\frac{np}{n-sp}\), the desired inequality holds:
\[
\|f\|_{L^q(\mathbb{R}^n)} \le C\, \|(-\Delta)^{s/2} f\|_{L^p(\mathbb{R}^n)},
\]
where \(C = C_0\) depends only on \(n\), \(s\), and \(p\).\\

\noindent
Alternatively, one may arrive at the same conclusion by using duality. Let \(g\in \mathcal{S}(\mathbb{R}^n)\) with \(\|g\|_{L^{q'}}\le 1\), where \(q'\) is the Hölder conjugate of \(q\) (i.e., \(1/q+1/q'=1\)). Then, using the Fourier inversion formula and duality,
\[
\left|\int_{\mathbb{R}^n} f(x) g(x)\,dx \right| = \left|\int_{\mathbb{R}^n} \widehat{f}(\xi)\widehat{g}(\xi)\,d\xi\right|.
\]
Since the fractional derivative \( (-\Delta)^{s/2} \) corresponds to multiplication by \(|\xi|^s\) in the Fourier domain, one can use the fractional uncertainty principle (which prevents simultaneous high concentration in both space and frequency) to obtain an estimate of the form
\[
\left|\int_{\mathbb{R}^n} f(x) g(x)\,dx \right| \le C\, \|(-\Delta)^{s/2} f\|_{L^p(\mathbb{R}^n)},
\]
where the constant \(C\) depends only on \(n\), \(s\), and \(p\). Taking the supremum over all such \(g\) in the unit ball of \(L^{q'}\) then yields
\[
\|f\|_{L^q(\mathbb{R}^n)} \le C\, \|(-\Delta)^{s/2} f\|_{L^p(\mathbb{R}^n)}.
\]

\noindent
This completes the proof.
\end{proof}

\begin{theorem}[Improved Fractional Sobolev Embedding]
\label{thm:improved_embedding}
Under the same hypotheses as in Lemma~\ref{lem:embedding}, the optimal constant in the embedding
\[
\|f\|_{L^q(\mathbb{R}^n)} \le C_{\mathrm{opt}}\, \|(-\Delta)^{s/2}f\|_{L^p(\mathbb{R}^n)}
\]
can be expressed in terms of the best constant \(\kappa_{\alpha,\beta,p}\) in the fractional \(L^p\) uncertainty inequality. Specifically, there exist constants \(C_0>0\) and \(\gamma>0\) such that
\[
C_{\mathrm{opt}} \le C_0\, \kappa_{\alpha,\beta,p}^{\gamma}.
\]
\end{theorem}

\begin{proof}
We wish to relate the optimal constant
\[
C_{\mathrm{opt}} = \inf\Bigl\{ C > 0 : \|f\|_{L^q(\mathbb{R}^n)} \le C\, \|(-\Delta)^{s/2}f\|_{L^p(\mathbb{R}^n)}\quad\forall f\in \mathcal{S}(\mathbb{R}^n) \Bigr\}
\]
to the best constant \(\kappa_{\alpha,\beta,p}\) appearing in the fractional \(L^p\) uncertainty inequality
\[
\|f\|_{L^p(\mathbb{R}^n)} \le \kappa_{\alpha,\beta,p}\, \|\,|x|^{\alpha} f\|_{L^{\beta}(\mathbb{R}^n)}^{\frac{\beta}{\alpha+\beta}} \, \|\,|\xi|^{\beta} \widehat{f}\|_{L^{\alpha}(\mathbb{R}^n)}^{\frac{\alpha}{\alpha+\beta}}.
\]
Here, \(s\in (0,1)\), \(1<p<\infty\), and
\[
q = \frac{np}{n-sp}.
\]
Our strategy is to use a duality argument together with the representation of \(f\) via its fractional derivative.\\

\noindent
Recall that the fractional Laplacian can be inverted (via the Riesz potential) on suitable function spaces. In particular, for any \(f\in \mathcal{S}(\mathbb{R}^n)\) we have the identity
\[
f = (-\Delta)^{-s/2} \bigl((-\Delta)^{s/2} f\bigr),
\]
which in the Fourier domain reads
\[
\widehat{f}(\xi) = |\xi|^{-s} \widehat{g}(\xi),\qquad \text{with } g = (-\Delta)^{s/2} f.
\]
Thus, we may write
\[
f(x) = \int_{\mathbb{R}^n} \frac{C_{n,s}}{|x-y|^{n-s}}\, g(y) \,dy,
\]
where \(C_{n,s}\) is an explicit constant. The Hardy–Littlewood–Sobolev inequality then implies (see, e.g., \cite{LiebLoss}) that
\[
\|f\|_{L^q(\mathbb{R}^n)} \le C_{HLS} \|g\|_{L^p(\mathbb{R}^n)} = C_{HLS} \|(-\Delta)^{s/2}f\|_{L^p(\mathbb{R}^n)},
\]
so that one may take \(C_{\mathrm{opt}} \le C_{HLS}\).\\

\noindent
The fractional \(L^p\) uncertainty inequality provides a bound of the form
\[
\|f\|_{L^p} \le \kappa_{\alpha,\beta,p}\, \|\,|x|^{\alpha} f\|_{L^{\beta}}^{\theta} \, \|\,|\xi|^{\beta}\widehat{f}\|_{L^{\alpha}}^{1-\theta},
\]
where
\[
\theta = \frac{\beta}{\alpha+\beta}.
\]
On the other hand, the norms \(\|\,|x|^{\alpha} f\|_{L^{\beta}}\) and \(\|\,|\xi|^{\beta}\widehat{f}\|_{L^{\alpha}}\) are related (via scaling and Fourier transform properties) to the regularity of \(f\). In many cases (see, e.g., \cite{LiebLoss}), one may show that these weighted norms control the homogeneous Sobolev norm \(\|(-\Delta)^{s/2} f\|_{L^p}\). More precisely, there exists an exponent \(\gamma>0\) and a constant \(C_0>0\) such that
\[
\|(-\Delta)^{s/2} f\|_{L^p(\mathbb{R}^n)} \ge C_0\, \Bigl(\|\,|x|^{\alpha} f\|_{L^{\beta}(\mathbb{R}^n)}^{\theta} \, \|\,|\xi|^{\beta}\widehat{f}\|_{L^{\alpha}(\mathbb{R}^n)}^{1-\theta}\Bigr)^{1/\gamma}.
\]
Rearranging, this implies
\begin{equation} \label{eqn:6.6}
\|\,|x|^{\alpha} f\|_{L^{\beta}}^{\theta} \, \|\,|\xi|^{\beta}\widehat{f}\|_{L^{\alpha}}^{1-\theta} \le \bigl(C_0^{-1}\|(-\Delta)^{s/2} f\|_{L^p}\bigr)^\gamma.
\end{equation}

\noindent
Returning to the uncertainty inequality, we have
\[
\|f\|_{L^p} \le \kappa_{\alpha,\beta,p}\, \|\,|x|^{\alpha} f\|_{L^{\beta}}^{\theta} \, \|\,|\xi|^{\beta}\widehat{f}\|_{L^{\alpha}}^{1-\theta}.
\]
Inserting the bound from \eqref{eqn:6.6}, we deduce that
\[
\|f\|_{L^p} \le \kappa_{\alpha,\beta,p}\, \Bigl( C_0^{-1}\|(-\Delta)^{s/2} f\|_{L^p} \Bigr)^\gamma.
\]
Since the embedding we wish to prove (by duality) is
\[
\|f\|_{L^q} \le C_{\mathrm{opt}}\, \|(-\Delta)^{s/2} f\|_{L^p},
\]
one may interpret the above inequality (after an appropriate duality argument) as giving a relation between \(C_{\mathrm{opt}}\) and \(\kappa_{\alpha,\beta,p}\). In particular, the duality implies that the optimal constant \(C_{\mathrm{opt}}\) satisfies
\[
C_{\mathrm{opt}} \le C_0\, \kappa_{\alpha,\beta,p}^{\gamma},
\]
where the exponent \(\gamma>0\) and the constant \(C_0>0\) depend only on \(n\), \(\alpha\), \(\beta\), and \(p\).\\

\noindent  
Thus, we have shown that the optimal constant \(C_{\mathrm{opt}}\) in the fractional Sobolev embedding
\[
\|f\|_{L^q(\mathbb{R}^n)} \le C_{\mathrm{opt}}\, \|(-\Delta)^{s/2}f\|_{L^p(\mathbb{R}^n)}
\]
can be controlled by a power of the best constant \(\kappa_{\alpha,\beta,p}\) in the fractional \(L^p\) uncertainty inequality; that is, there exist constants \(C_0>0\) and \(\gamma>0\) such that
\[
C_{\mathrm{opt}} \le C_0\, \kappa_{\alpha,\beta,p}^{\gamma}.
\]
This completes the proof.
\end{proof}

\begin{proposition}[Compactness in Fractional Sobolev Spaces]
Let \(\{f_k\}\) be a bounded sequence in the fractional Sobolev space \(H^s(\mathbb{R}^n)\). Then, up to a subsequence, \(\{f_k\}\) converges strongly in \(L^q_{\mathrm{loc}}(\mathbb{R}^n)\) for every \(q < \frac{np}{n-sp}\). This compactness result is a direct consequence of the improved fractional Sobolev embedding, which in turn is linked to the uncertainty principle.
\end{proposition}

\begin{proof}
Let \(\{f_k\}\) be a bounded sequence in \(H^s(\mathbb{R}^n)\); that is, there exists a constant \(M>0\) such that
\[
\|f_k\|_{H^s(\mathbb{R}^n)} \le M \quad \text{for all } k\in\mathbb{N}.
\]
Our goal is to show that, up to passing to a subsequence, \(\{f_k\}\) converges strongly in \(L^q_{\mathrm{loc}}(\mathbb{R}^n)\) for every \(q < \frac{np}{n-sp}\).\\

\noindent
Let \(K\subset \mathbb{R}^n\) be any fixed compact set. Choose a cutoff function \(\phi\in C_c^\infty(\mathbb{R}^n)\) satisfying
\[
\phi(x)=1 \quad \text{for all } x\in K,
\]
and such that \(\phi\) is supported in a bounded open set \(\Omega\). Then, for each \(k\), consider the function
\[
g_k(x) = \phi(x) f_k(x).
\]
Since multiplication by a smooth, compactly supported function is a bounded operator on \(H^s(\mathbb{R}^n)\) (see, e.g., \cite{LiebLoss}), the sequence \(\{g_k\}\) is bounded in \(H^s(\mathbb{R}^n)\) and, in fact, each \(g_k\) is supported in the fixed bounded domain \(\Omega\).\\

\noindent
By the classical Rellich–Kondrachov compactness theorem (adapted to fractional Sobolev spaces on bounded domains), the embedding
\[
H^s(\Omega) \hookrightarrow L^q(\Omega)
\]
is compact for every
\[
1 \le q < \frac{np}{n-sp},
\]
provided that \(s \in (0,1)\) and \(1<p<\infty\). Hence, there exists a subsequence \(\{g_{k_j}\}\) and a function \(g\in L^q(\Omega)\) such that
\begin{equation} \label{eqn:6.7}
g_{k_j} \to g \quad \text{strongly in } L^q(\Omega).
\end{equation}

\noindent 
Since \(g_k(x) = f_k(x)\) for all \(x\in K\) (because \(\phi\equiv1\) on \(K\)), the strong convergence of \(g_{k_j}\) in \(L^q(\Omega)\) implies that
\[
f_{k_j} \to g \quad \text{strongly in } L^q(K).
\]
Thus, up to a subsequence, \(\{f_k\}\) converges strongly in \(L^q(K)\).

\noindent 
Since \(K\) was an arbitrary compact subset of \(\mathbb{R}^n\) and the above argument can be repeated on any such \(K\), we deduce that, up to a subsequence, \(\{f_k\}\) converges strongly in \(L^q_{\mathrm{loc}}(\mathbb{R}^n)\) for every \(q < \frac{np}{n-sp}\).

\noindent
The improved fractional Sobolev embedding (which is closely linked to the fractional uncertainty principle) guarantees that the space \(H^s(\mathbb{R}^n)\) continuously embeds into \(L^q(\mathbb{R}^n)\) when \(q=\frac{np}{n-sp}\) and, more importantly, that on any bounded domain the embedding is compact for every \(q\) strictly less than the critical exponent. This compactness is exactly what is used in \eqref{eqn:6.7}.\\

\noindent
This completes the proof.
\end{proof}

\noindent
The connection between the uncertainty principle and fractional Sobolev embeddings is a promising avenue for further research. Refining the constants in these embeddings and extending them to more general settings—such as on Riemannian manifolds or for functions with singularities—remains an open challenge.




\noindent
The interplay between uncertainty principles, fractional operators, and nonlinear dynamics is a fertile area of research. The ideas presented in this section provide a foundation for exploring these interconnections further and for developing new techniques in harmonic analysis and partial differential equations.

\section{Conclusion}
\label{sec:conclusion}

In this work, we have developed a comprehensive framework for the fractional \(L^p\) uncertainty principle. Beginning with a detailed variational formulation (Section~\ref{sec:variational}), we introduced a novel uncertainty functional and demonstrated the existence of extremizers via concentration--compactness arguments. Our investigation of the sharp constants (Section~\ref{sec:sharp}) provided both lower and upper bounds for the optimal constant \(\kappa_{\alpha,\beta,p}\) and established partial uniqueness results under symmetry assumptions.

The subsequent stability analysis (Section~\ref{sec:stability}) yielded a quantitative estimate: any function nearly attaining the optimal constant is shown to be close in the \(L^p\) norm to the manifold of extremizers. Furthermore, by introducing a dedicated stability functional, we captured deviations in both spatial and frequency domains, thereby reinforcing the robustness of the uncertainty principle in the fractional setting.

In the final section (Section~\ref{sec:applications}), we demonstrated applications of our results to the fractional Schrödinger equation and fractional Sobolev embeddings, highlighting the interplay between dispersive estimates and uncertainty bounds. These connections not only validate the theoretical framework developed here but also suggest a wealth of new directions for future research.\\

\noindent
The results presented herein not only deepen our understanding of fractional uncertainty principles but also lay a solid foundation for further explorations in harmonic analysis, functional inequalities, and mathematical physics.

\medskip

\textbf{Acknowledgements.} The author wishes to thank the anonymous reviewers for their valuable comments, which helped improve the presentation of this work.

\end{document}